\documentclass[reqno,11pt]{amsart}
\usepackage{amsmath,amssymb,mathrsfs,amsthm,amsfonts}
\usepackage[inline]{enumitem} 
\usepackage[usenames,dvipsnames]{xcolor}
\usepackage{hyperref}
\usepackage{comment}
\usepackage{array,longtable,booktabs,caption}
\DeclareMathAlphabet{\mathpzc}{OT1}{pzc}{m}{it}
\usepackage{stmaryrd}

\hypersetup{%
	colorlinks=true, linkcolor=blue,
	citecolor=ForestGreen
}
\usepackage[paper=letterpaper,margin=1.2in, top=1in]{geometry}
\usepackage{graphicx}
\usepackage{mathrsfs}
\usepackage{listings}

\usepackage{bbm}
\usepackage{tikz}
\usetikzlibrary{arrows}

\newtheorem{thm}{Theorem}[section]
\newtheorem{cor}[thm]{Corollary}
\newtheorem{prop}[thm]{Proposition}
\newtheorem{lem}[thm]{Lemma}

\theoremstyle{definition}
\newtheorem{defn}[thm]{Definition}

\newtheorem{rk}[thm]{Remark}

\numberwithin{equation}{section}

\newcommand{\Ex}{\mathbb{E}}

\newcommand{\Con}{\mathrm{C}}
\newcommand{\mE}{\mathbf{E}}

\newcommand{\e}{\varepsilon}

\title[Erratic KPZ limits via moments]{
Moment-based approach for two erratic KPZ scaling limits
}
\author{Shalin Parekh}

\begin{document}

\begin{abstract}A recent paper of Tsai \cite{Ts24} shows how the first few moments of a stochastic flow in the space of measures can completely determine its law. Here we give another proof of this result for the particular case of the one-dimensional multiplicative stochastic heat equation (mSHE), and then we investigate two corollaries. 
The first one recovers a recent result of Hairer on a ``variance blowup" problem related to the KPZ equation \cite{Hai24}, albeit in a much weaker topology. The second one recovers a KPZ scaling limit result related to random walks in random environments from \cite{Par24}, but in a weaker topology. In these two problems, we furthermore explain why it is hard to directly use the martingale characterization of the mSHE, the chaos expansion, or other known methods. Using the moment-based approach avoids technicalities, leading to a short proof.
\end{abstract}

\maketitle
\vspace{-0.3 in}

\section{Introduction}

The (1+1)-dimensional multiplicative-noise stochastic heat equation is a stochastic partial differential equation (SPDE) written formally as  \begin{equation}\partial_t \mathcal U(t,y) = \frac12 \partial_y^2 \mathcal U(t,y) + \mathcal U(t,y)\xi(t,y) ,\;\;\;\;\;\;\; t\ge 0, y\in \Bbb R, \label{she}\tag{mSHE}
    \end{equation}
where $\mathcal U(0,y) = \mathcal U_0(y)$ and $\xi$ is a Gaussian space-time white noise, that is, the Gaussian field with covariances $\Bbb E[\xi(t,x)\xi(s,y)] =\delta(t-s)\delta(x-y)$. Such $\xi$ will be a tempered distribution that is not a function, thus one needs some work to make sense of this object. Using stochastic calculus it is possible to define a natural notion of solution started from any nonnegative measure-valued initial condition $\mathcal U_0$ of at-worst-exponential growth at infinity, see \cite{Wal86}. This is called the It\^o solution or It\^o-Walsh solution of the SPDE given by \eqref{she}. 

It turns out that the It\^o-Walsh solution $\mathcal U$ is a random function that remains strictly positive at all positive times almost surely, as long as the initial data $\mathcal U_0$ is a non-negative and nonzero Borel measure. This random function $\mathcal U$ $-$ and its logarithm, which solves the so-called Kardar-Parisi-Zhang (KPZ) equation $-$ has been the subject of intense research in recent decades due to its connection to a number of active areas, including exactly solvable systems, interface growth models, interacting particle systems, directed polymers, KPZ universality, and stochastic analysis. We refer to \cite{FS10, Qua11, Cor12, QS15, CW17, CS20} for some surveys on this topic.

A recent work of Tsai \cite{Ts24} studies the (2+1)-dimensional analogue of \eqref{she}, which is more difficult to define because it is not known to be the strong solution of an SPDE like the (1+1)-dimensional case, see \cite{CSZ_2d, CSZ23b}. The Tsai paper shows that the (2+1)-dimensional analogue of \eqref{she} is \textit{uniquely characterized in law as a measure-valued stochastic flow using only its first four moments.} That result is remarkable in the sense that it makes rigorous various ``moment-based replica tricks" which physicists have long used, but which were thought to be non-rigorous from a mathematical perspective due to the extremely fast growth of the moments. Inspired by \cite{Ts24}, we use a similar methodology for the \textit{one-dimensional} object \eqref{she} in order to study two ``unusual" problems in KPZ scaling limit theory that exhibit very erratic behavior. The two corresponding main results are Theorems \ref{mr2} and \ref{mr3} below. 

First let us introduce some notations. We let $\mathcal M(\Bbb R^2)$ denote the set of locally finite and non-negative Borel measures on $\Bbb R^2$. We equip $\mathcal M(\Bbb R^2)$ with the vague topology, that is, the weakest topology under which the maps $\mu \mapsto \int_{\Bbb R^2} f\;d\mu$ are continuous, for all $f\in C_c^\infty(\Bbb R^2)$. We also denote $\Bbb R^4_{\uparrow}:= \{ (s,t,x,y) \in \Bbb R^4: s<t\}$, and $C(\Bbb R^4_{\uparrow})$ the set of functions from $\Bbb R^4_\uparrow\to\Bbb R$ that are continuous. Equip $C(\Bbb R^4_{\uparrow})$ with the topology of uniform convergence on compacts. $\Bbb Q$ will denote rational numbers, but in all results and proofs it could be replaced by any countable dense subset of $\Bbb R$ as desired. The following result is due to Tsai \cite{Ts24}, and we state a less optimized (1+1)-dimensional version of it that will suffice for our purposes.

\begin{prop}[(1+1)-dimensional variant of Theorem 1.3 in \cite{Ts24}]\label{mr}
    Let $Z_{s,t}$ be a family of $\mathcal M(\Bbb R^2)$-valued random variables indexed by $-\infty <s\le t<\infty$ with $s,t\in\Bbb Q,$ all defined on some complete probability space $(\Omega,\mathcal F,\Bbb P)$. Suppose that the following are true: 
    \begin{enumerate}

    \item $Z_{s_j,t_j}$ are independent under $\Bbb P$ whenever $(s_j,t_j)$ are disjoint intervals.

        \item For all $f\in C_c^\infty(\Bbb R^2)$ and all $\psi\in C_c^\infty(\Bbb R)$ with $\int_\Bbb R\psi=1$, we have that $$\int_{\Bbb R^4 } \delta^{-1}\psi(\delta^{-1}(y-w))\cdot f(x,z) Z_{t,u}(dy,dz) Z_{s,t}(dx,dw)  \to \int_{\Bbb R^2} f(x,z) Z_{s,u}(dx,dz)$$ in probability as $\delta\to 0$, whenever $s<t<u$.

        \item For $n\in \Bbb N$ and $\phi,\psi\in C_c^\infty(\Bbb R)$ we have that \begin{align*}\Bbb E\bigg[ \bigg( \int_{\Bbb R^2}& \phi(x) \psi(y)Z_{s,t}(dx,dy) \bigg)^n \bigg] \\&= \int_{\Bbb R^{n}} \prod_{k=1}^n \phi(x_k) \mathbf E_{\mathrm{BM}^{\otimes n}}^{(x_1,...,x_n)} \bigg[\prod_{k=1}^n \psi(B^k_{t-s}) e^{\sum_{1\le i<j\le n} L_0^{B^i-B^j}(t-s)} \bigg] dx_1 \cdots dx_n.
        \end{align*}
        The expectation on the right side is with respect to a $n$-dimensional Brownian motion $(B^1,...,B^n)$ started from $(x_1,...,x_n)$, and $L_0^X$ denotes the local time at 0 of the semimartingale $X$.
    \end{enumerate}
    Then 
    there exists a $C(\Bbb R^4_\uparrow)$-valued random variable $\tilde Z$ defined on the same probability space, such that $(Z_{s,t},f) = \int_{\Bbb R^2} f(x,z) \tilde Z_{s,t}(x,z)dxdz$ almost surely, for all $(s,t)\in\Bbb Q^2$ (with $s<t$) and $f\in C_c^\infty(\Bbb R^2)$. 
    
    Furthermore, there exists a space-time white noise $\xi$ on $\Bbb R^2$, defined on the same probability space, with the property that $\tilde Z_{s,t}$ are the \textbf{propagators} of \eqref{she} driven by $\xi$: for all $s,x\in \Bbb R$ the function $(t,y)\mapsto \tilde Z_{s,t}(x,y)$ is a.s. the It\^o-Walsh solution of the equation \eqref{she} started from time $s$ with initial condition $\delta_x(y)$. In particular, the finite-dimensional marginal laws $(Z_{s_1,t_1},...,Z_{s_m,t_m})$ are $\mathrm{uniquely\; determined}$ by Conditions (1) - (3), for any finite collection of indices $s_j\le t_j$.
\end{prop}

Although the above result is not the main focus of this paper, we will give a novel proof of it in Section \ref{s2} for the sake of completeness. However, our proof is specific to the one-dimensional case and it is not as robust as the proof in \cite{Ts24}. The proof is based on stochastic calculus techniques rather than the Lindeberg replacement method of \cite{Ts24}, and 
shows that these three conditions actually imply the \textit{martingale problem} for the equation \eqref{she} \cite{SV, KS88, BG97}, and also that the noise can be reconstructed via a stochastic integral formula. In Condition (3), we only need the equality to hold for $1\le n \le 19$ as opposed to all $n\in \Bbb N$, as the proof will show (see Remark \ref{wm}). Optimally it should be 4 moments, as shown for the (2+1)-dimensional analogue in \cite{Ts24}. 
The weaker version above has a short proof and suffices for the two main problems that we want to study, see Theorems \ref{mr2} and \ref{mr3} below.


From the perspective of scaling limit theory, the value of a result like Proposition \ref{mr} is that it can be used to identify limit points quite easily in cases where the moments are known. This now leads to the question of whether or not there are any interesting problems in scaling limit theory that can be solved using the above approach, which cannot be done directly using some other well-established methods such as the martingale problem for the equation \eqref{she}. The following theorem was first proved by Hairer in \cite{Hai24}, and we now recover the result using Proposition \ref{mr}, albeit in a much weaker topology.

\begin{thm}[Scaling limit of the SHE with derivative noise]\label{mr2}
    Fix some $p\in\Bbb N$. Consider the SPDE 
    \begin{equation}\label{dshe}\tag{dSHE}\partial_t z^\epsilon(t,y) = \tfrac12 \partial_y^2 z^\epsilon(t,y) + \epsilon^{p-\frac14} z^\epsilon(t,y) \partial_y^p\eta^\epsilon(t,y),\;\;\;\;\;\;\;\;\;\;\; t\ge 0, y\in \Bbb R,
    \end{equation}where $\eta^\epsilon:= \varphi_\epsilon *\eta$ for a standard Gaussian space-time white noise $\eta$ on $\Bbb R_+\times\Bbb R$ and where the convolution is in space only, and where $  \varphi_\epsilon (y):= \epsilon^{-1} \varphi(\epsilon^{-1}y)$ for some fixed smooth even mollifier $\varphi \in C_c^\infty(\Bbb R)$.

    Let $Z^\epsilon_{s,t}(x,y)$ denote the associated family of propagators given by the It\^o-Walsh solution $z^\epsilon$ to the equation \eqref{dshe}. Then $Z^\epsilon$ converges in law as $\epsilon\to 0$ to the family of propagators given by the It\^o-Walsh solution of the SPDE $$\partial_t \mathcal U(t,y) = \tfrac12 \partial_y^2 \mathcal U(t,y) +  \sigma_p \mathcal U(t,y) \xi(t,y), \;\;\;\;\;\;\;\;\;\;\;\;\sigma_p^2 := \int_{\Bbb R} (\varphi * \varphi^{(2p-1)}(x))^2 dx$$ where $\varphi^{(k)}$ is the $k^{th}$ derivative and $\xi$ is another space-time white noise. Convergence is in the sense of finite-dimensional laws $(Z^\epsilon_{s_1,t_1},...,Z^\epsilon_{s_m,t_m}),$ where each coordinate has the topology of $\mathcal M(\Bbb R^2)$, for any finite collection of indices $s_j\le t_j$.
\end{thm}

We prove Theorem \ref{mr2} in Section \ref{s3}. We remark that the result of \cite{Hai24} was on a circle as opposed to ours on the full line, and moreover it was restricted to more regular initial data. However, that result was stronger in many ways. First of all, that result can consider \textit{non-Markovian approximations,} such as mollifications of $\eta$ in both space and time. But more crucially, that paper proves convergence in law with a much stronger topology of H\"older continuous functions as opposed to just measure-valued stochastic flows. This question of topology is very important here, because one of the most striking aspects of Proposition \ref{mr} is the fact that $Z_{s,t}$ is not assumed to be very regular at all, in particular there is no \textit{a priori} assumption that it is continuous or even function-valued. This makes tightness fairly easy to prove in our case. In contrast, obtaining tightness of \eqref{dshe} in a stronger topology seems very difficult to do directly, because of the fact that there is a derivative 
when we try to calculate moments of local H\"older norms by writing out the Duhamel formula. 

It is not possible to improve Theorem \ref{mr2} so that convergence occurs a.s. or in probability. Let us explain why this is the case, and why it is difficult to directly use established methods such as the martingale problem or the chaos expansion to solve this problem. A direct application of the martingale problem would require us to prove tightness in a space of continuous functions, which as pointed out above is difficult. But even then, obtaining nice expressions for the quadratic variations of the relevant martingales seems to be extremely tedious. On the other hand, the chaos expansion method of \cite{AKQ14a, CSZ17a} would be difficult because as pointed out in \cite{Hai24}, the ``correct driving noise" in \eqref{dshe} to see the limit is \textbf{not} $\eta^\epsilon$ itself, but rather it is $\epsilon^{2p-\frac12} ((\partial_xK \star \partial_y^p \eta^\epsilon)^2 -C_\epsilon)$ for some large constant $C_\epsilon \uparrow \infty,$ where $K$ is the heat kernel and $\star$ is space-time convolution. To form a chaos expansion with \textit{this} noise seems to be a gargantuan task, since it is non-Gaussian and has nonzero correlations in both space and time. 

The paper \cite{Hai24} proves the result by using the theory of BPHZ renormalization \cite{HS23, CH18, Bai} to automatically find canonical \textit{probabilistic lifts} of the driving noise (which consist of some large but finite number of homogeneous chaoses over $\eta^\epsilon$), together with the \textit{analytic theory} of regularity structures \cite{Hai14} to construct a continuous map from the lifted noise to the solution. Those structural properties of the solution map are extremely powerful, and the reason why the result of \cite{Hai24} is more robust. Hairer calls the scaling of Theorem \ref{mr2} an example of ``variance blowup" because if one naively tries to set up the regularity structure based on just $\eta^\epsilon$ without \textit{separately} keeping track of $\epsilon^{2p-\frac12} ((\partial_xK \star \partial_y^p \eta^\epsilon)^2 -C_\epsilon)$, then the stochastic part of the construction of some of the trees fails to converge in $L^2$ due to an unbounded variance (despite the \textit{scaling subcriticality} of the equation for $z^\epsilon$ in the sense of \cite{Hai14}), similar to what happens if one tries to construct iterated integrals with respect to fractional Brownian motion with $H\le 1/4$. These variance blowup issues cause \cite{Hai24} to have to slightly modify the BPHZ theory before applying it, which makes it somewhat technical compared to the more elementary approach here.

Next, we consider another application of Proposition \ref{mr}, which is unrelated to Theorem \ref{mr2} despite the apparent similarity. The following result was first conjectured in \cite{dom} and then proved in \cite{Par24}. 

\begin{thm}[Scaling limit of the SHE with advective noise]\label{mr3}
    Let $\eta$ be a space-time white noise on $\Bbb R_+\times\Bbb R$, and let $\eta^\epsilon:= \varphi_\epsilon *\eta$ where the convolution is in space only, and where $  \varphi_\epsilon (y):= \epsilon^{-1} \varphi(\epsilon^{-1}y)$ for some fixed smooth even mollifier $\varphi \in C_c^\infty(\Bbb R)$ such that $\|\varphi\|_{L^2(\Bbb R)}<1$ (equivalently $\sup_{x\in \Bbb R} \varphi*\varphi(x)<1$). 
    
    Denote by $\mathcal Z^\epsilon_{s,t}(x,y)$ the propagators given by the It\^o-Walsh solution of the SPDE \begin{equation}\label{1a}\tag{aSHE}\partial_t u(t,y) = \frac12 \partial_y^2 u(t,y) + u(t,y)\eta^\epsilon(t,y) + \epsilon^{1/2}\partial_y\big(u(t,y)\eta^\epsilon(t,y)\big),\;\;\;\;\;\;\;\; t\ge 0,y\in\Bbb R.
    \end{equation}
Then $\mathcal Z^\epsilon$ converges in law as $\epsilon\to 0$ to the family of propagators given by the It\^o-Walsh solution of $$\partial_t \mathcal U(t,y) = \tfrac12 \partial_y^2 \mathcal U(t,y) +  \gamma_{\mathrm{ext}} \mathcal U(t,y) \xi(t,y), \;\;\;\;\;\;\;\;\;\;\gamma_{\mathrm{ext}}^2 := \int_\mathbb R\frac{\varphi*\varphi(a)}{1- \varphi*\varphi(a)}da.$$ Here $\xi$ is another space-time white noise. The convergence is in the sense of finite-dimensional laws $(\mathcal Z^\epsilon_{s_1,t_1},...,\mathcal Z^\epsilon_{s_m,t_m}),$ where each coordinate has the topology of $\mathcal M(\Bbb R^2)$.
\end{thm}

We prove Theorem \ref{mr3} in Section \ref{s4}. The scaling of the equation \eqref{1a} exhibits even more erratic properties than the one in Theorem \ref{mr2}, as we will explain just below. Nonetheless, it is physically relevant due to its appearance in problems related to the topic of KPZ universality in models of random walk in random environments, see e.g. \cite{ldt, ldt2, bld, bc, gu, mark, dom, DDP23, DDP23+, Par24} for more discussion. In much greater generality it was first proved in \cite{Par24} that as $\epsilon\to 0$, \eqref{1a} converges in law to the solution of \eqref{she}, in the stronger topology of $C([0,T],\mathcal M(\Bbb R))$. See Section 6.3 of that paper for the particular case of \eqref{1a}. That result did not consider the full family of propagators, but rather just $s=x=0$. Combined with Theorem \ref{mr3}, that result now extends to the full family of propagators.

As with Theorem \ref{mr2}, the convergence of Theorem \ref{mr3} does not occur a.s. or in probability, and in fact extra noise will be generated in the limit that is independent of the original driving noise $\eta$ (as is suggested by the fact that $\gamma_{\mathrm{ext}}> 1$). This was explained in detail in \cite[Theorem 1.4]{DDP23+}, which studied a discrete version of this model \eqref{1a}. A related pathological phenomenon here is that (unlike Theorem \ref{mr2}) convergence of $\mathcal Z^\epsilon$ will \textbf{not} occur in a space of continuous functions such as $C(\Bbb R^4_\uparrow)$. This is not just a technicality. Indeed, if one assumes that such a convergence holds in the space of continuous functions, then one can derive a contradiction, see the discussion in \cite{DDP23+} for a discrete analogue.

This suggests that unlike Theorem \ref{mr2}, the scaling limit problem in Theorem \ref{mr3} cannot even be studied using regularity structures (at least not directly). Indeed, one can check that the equation of Theorem \ref{mr3} is not \textit{scaling subcritical} in the sense of \cite{Hai14}. To see that, one can 
consider the ``naive" regularity structure where the model space $T$ is generated by two abstract noise symbols $\Xi_1,\Xi_2$ and an integration operator $\mathcal I$ of order 2 (representing convolution with the heat kernel in space and time). The rule for $T$ is that $1 \in T$, and 
$\tau \in T$ implies $\mathcal I(\tau \Xi_1) + \mathcal I'(\tau \Xi_2) \in T$. Here $\Xi_1$ stands for $\eta^\epsilon$ and $\Xi_2$ stands for $\epsilon^{1/2} \eta^\epsilon$. One wants to solve a fixed point problem of the form $x = \mathcal I(x\Xi_1) + \mathcal I'(x\Xi_2)$, where $\Xi_1$ would have order $-3/2$ while $\Xi_2$ would have order $-1$ thanks to the multiplying factor $\epsilon^{1/2}$. Now the reason that local subcriticality fails is that orders of products are the sum of the orders, thus $\Xi_2 \mathcal I'(\Xi_2)$ has order $-1$ again, thus $\Xi_2 \mathcal I'(\Xi_2 \mathcal I'(\Xi_2))$ has order $-1$ once again, and so on. Thus it is a \textit{scaling critical model}, since the regularity remains bounded below but fails to improve under successive iterations of the rules for this regularity structure. Thus Theorem \ref{mr3} is not a ``variance blowup" problem in the same sense as described in \cite{Hai24}, because the problem occurs at the level of the \textit{analytic theory} and not just the \textit{probabilistic lifts}. Still, it exhibits a similar flavor in the sense that putting the vanishing factor $\epsilon^{1/2}$ on $\eta^\epsilon$ appearing in the advective term leads to a nontrivial limit in distribution but not pathwise. 



The chaos expansion also fails for the equation of Theorem \ref{mr3}, due to some of the $L^2$-mass of $\mathcal Z^\epsilon$ escaping into the tails of the $\eta^\epsilon$-chaos expansion (very high-order chaos) as $\epsilon\to 0$, see the discussion in the introduction of \cite{DDP23+}. In \eqref{1a}, the ``correct prelimiting noise" does not live in any fixed-order chaos as was the case for Theorem \ref{mr2}. \textit{Conjecturally} it should be $\eta^\epsilon \cdot F^\epsilon$ where $F^\epsilon$ is the space-time stationary solution to $\partial_t v  = \frac12 \partial_y^2 v + \epsilon^{1/2} \partial_y(v\eta^\epsilon)$. More intuitively in terms of the regularity structure discussion, $F=\sum_{k=0}^\infty(\mathcal I'\Xi_2)^{\circ k}$ where $(\mathcal I'\Xi_2)^{\circ k}$ denotes the $k$-fold iteration of the operation $\tau\mapsto \mathcal I'(\tau\Xi_2)$. 
Ultimately in \cite{Par24} and related problems \cite{DDP23,DDP23+}, the method of \textit{martingale problem} was successfully used for this type of model, but the resultant effort was long and technical. Here the result is weaker, but nonetheless we can successfully identify the limit points \textit{uniquely} with a calculation of only five pages, thus illustrating the power of the weaker moment-based formulation above.

Having discussed the benefits, let us discuss some of the major limitations of Proposition \ref{mr} in (1+1)-dimensional KPZ-related problems. First, there are many interesting systems (for example the WASEP of \cite{BG97}) in which calculating the moments is not so easy, where the classical martingale problem would actually be simpler than the moment method applied here. Second, one always suspects tightness in a better topology than just finite-dimensional laws in $\mathcal M(\Bbb R^2)$, but the moment-based approach of Proposition \ref{mr} by itself will never be able to show the required tightness in any better topology, thus illustrating another limitation (although 
the moment method \textit{does} identify the limit point uniquely if the tightness has already been shown in a stronger topology such as $C(\Bbb R^4_\uparrow)$). 
Third, the method of Proposition \ref{mr} is limited to systems that are \textit{linear as a function of their initial data}, see the discussion in Subsection 2.1. Fourth, Proposition \ref{mr} is unable to calculate\textbf{ joint} limits in distribution. For example in the setting of Theorem \ref{mr2}, the method of \cite{Hai24} actually shows that $(\eta^\epsilon,Z^\epsilon)$ converge in law jointly to two independent processes with the correct marginals. In the setting of Theorem \ref{mr3}, the result of \cite[Theorem 1.4]{DDP23+} shows how to calculate the \textit{joint} limit in distribution of $(\eta^\epsilon,\mathcal Z^\epsilon)$ which has a nontrivial answer. Proposition \ref{mr} cannot do this.
\\
\\
\textbf{Outline.} In Section \ref{s2}, we give a short proof of Proposition \ref{mr}. In Section \ref{s3}, we prove Theorem \ref{mr2}. In Section \ref{s4}, we prove Theorem \ref{mr3}.
\\
\\
\textbf{Acknowledgements.} I thank Yu Gu for suggesting the problem of trying to recover the KPZ result from \cite{Hai24} using a different approach, and for interesting discussions about this problem. I thank Zuodi Xie for finding many typos and giving useful feedback on the proofs from a preliminary draft of this paper. I thank the two anonymous referees for very helpful feedback,
and for finding typos. I acknowledge support by the NSF MSPRF (DMS-2401884).

\section{Proof of Proposition \ref{mr}}\label{s2}

\subsection{Elementary discussion} Before proving the proposition, we first explain the main idea by considering a toy model given by the ``geometric Brownian flow," which is the same example given in \cite{Ts24}. 
\begin{prop}\label{gbf}
Consider a two-variable process $(G_{s,t})_{-\infty <s\le t<\infty}$ with $s,t\in \Bbb Q$, defined on some complete probability space, and satisfying 
\begin{enumerate}
        \item $G_{s_j,t_j}$ are independent if $(s_j,t_j)$ are disjoint intervals.

        \item $G_{s,t}G_{t,u}= G_{s,u}$ whenever $s<t<u$.

        \item For $n=1,2,3,4,5,6$, we have that $\Bbb E\big[ G_{s,t}^n \big] = e^{\frac{n^2}2(t-s)}.$
\end{enumerate}
These three conditions uniquely determine the finite-dimensional distributions of the process $(G_{s,t})_{s<t}$ as being the same as those of $(e^{B(t)-B(s)})_{s<t}$  for a standard Brownian motion $B$, and moreover $B$ can be defined on the same space so that the relation holds pathwise on $\Bbb Q^2$. 
\end{prop}
\begin{proof}
\cite{Ts24} proves this using the Lindeberg exchange principle, so let us apply a different method. Define the process $X_t:= e^{-\frac{t}2} G_{0,t}.$ First we note that $X$ admits a continuous modification, meaning that there exists a $C[0,T]$-valued random variable $\tilde X$ defined on the same probability space such that $\tilde X_t=X_t$ almost surely for all $t\in\Bbb Q\cap [0,T]$. This is because one has the bound $$\Bbb E[ |G_{0,t} - G_{0,s}|^3]^{1/3} = \Bbb E[ |G_{0,s}|^3 |G_{s,t}-1|^3 ]^{1/3} \le \Bbb E[ |G_{0,s}|^6]^{1/6} \Bbb E[|G_{s,t} -1|^{6}]^{1/6} \leq C_T|t-s|^{1/2}, $$ uniformly over $s,t\in [0,T]\cap \Bbb Q$ for any $T>0$. Now note that $\tilde X_t$ is a martingale using (1), (2), and (3) with $n=1$. For $t>0$, we claim that for any deterministic sequence of partitions of $[0,t]$ with mesh tending to 0 and having rational endpoints, one has
\begin{equation}\label{qvlim}\lim_{N\to\infty} \Bbb E\bigg[ \bigg( \sum_{i=1}^N ( X_{t_{i+1}}- X_{t_i})^2 - \sum_{i=1}^N X_{t_i}^2 (t_{i+1}-t_i)\bigg)^2\bigg] = 0. 
\end{equation}
To prove this, note that the claim is true\footnote{For any continuous $L^2$-martingale, the partition approximations of the quadratic variation process always converge \textit{in probability} to the quadratic variation process \cite[Theorem (1.8) in Chapter IV]{RY99}. Under the stronger assumption that the martingale is in $L^p$ for some $p>4$, this can be automatically upgraded to $L^2$ convergence by a uniform integrability argument, using e.g. the lower bound of BDG and then Doob's inequality. This will be used later as well.} if $X$ was replaced by the process $e^{B(t)-\frac{t}2}$. But since the first four moments of \textit{every multiplicative increment} of the latter process match those of $X$, the claim is true for $X$ as well. In particular the quadratic variation is given by $\langle \tilde X\rangle_t = \int_0^t \tilde X_s^2 ds.$ But if $\tilde X_t$ and $\tilde X_t^2 - \int_0^t \tilde X_s^2 ds$ are both \textit{continuous} martingales, this uniquely identifies $\tilde X$ in law, as the solution of the SDE given by $d\tilde X= \tilde XdB$ for a standard Brownian motion $B,$ see \cite[Proposition (2.4) in Chapter VII]{RY99} or the original treatise on martingale problems \cite{SV}. Note that $\tilde X$ is then strictly positive, and thus the Brownian motion $B$ can be defined on the same probability space as $X$, via the stochastic integral $B(t) = \int_0^t \tilde X_s^{-1} d\tilde X_s$. The initial data for the SDE is 1 since $\Bbb E[ (X_t-1)^2 ]\to 0$ as $t\to 0.$ This easily yields $\tilde X_t = e^{B(t)-\frac{t}2}$, so that $G_{s,t} = e^{B(t)-B(s)}$ almost surely for all $s,t\in\Bbb Q$, completing the proof.
\end{proof}

The proof can be modified slightly so that the assumption of six moments can be relaxed to four moments in Item (3), but then one needs to work with discontinuous processes, see Remark \ref{wm}. We also remark that in order to be able to use this method, it is important for the object of interest to be a \textit{linear function of its initial data}. This type of moment-based approach would never work for more complicated nonlinear multiplicative forcing terms, such as SDEs of the form $dX=\sigma(X)dB$ with nonlinear $\sigma$. This is apparent from the fact that $X_{t_i}^2$ would get replaced by some non-quadratic function in \eqref{qvlim} if one were to try to use this type of moment method for a more complex SDE, and then the moments would certainly \textit{not} suffice unless $\sigma$ is polynomial in $\sqrt x$. But more importantly, unless $\sigma(x) = ax$ for some $a\in \Bbb R$, the flow of maps associated to such an SDE would not respect \textit{multiplication} in the first place, as in Item (2) above. Very similar limitations apply in the SPDE case of Proposition \ref{mr}, for instance this method is not applicable if one wishes to recover results such as the one in \cite{BR} since the limit there is an additive SPDE.

\subsection{Proof of the result} With this intuition established, we now proceed to prove Proposition \ref{mr}, using a somewhat different method than the original work \cite{Ts24}. We always denote by $(Z_{s,t},f):= \int_{\Bbb R^2} f(x,y) Z_{s,t}(dx,dy)$. 
We also denote $(\phi,\psi)_{L^2(\Bbb R)}:= \int_{\Bbb R} \phi\psi$. The proof of Proposition \ref{mr} will follow immediately from the results of Proposition \ref{mr0}, Corollary \ref{solves}, and Proposition \ref{jnoise} below. 

\begin{prop}\label{mr0}
    In the setting of Proposition \ref{mr}, there exists a random four-parameter function $\tilde Z_{s,t}(x,y)$ with $-\infty<s<t<\infty$ and $x,y\in \Bbb R$, H\"older continuous on compact sets of $\{ (s,t,x,y): s<t\}$ and satisfying the following properties: 
    \begin{enumerate}\item $(Z_{s,t},f) = \int_{\Bbb R^2} \tilde Z_{s,t}(x,y)f(x,y)dxdy$ almost surely for all $s,t\in\Bbb Q \;(s<t)$ and all $f\in C_c^\infty(\Bbb R^2)$. In particular, $\tilde Z$ is defined on the same probability space. \item If $(s_j,t_j)$ are disjoint intervals then $\tilde Z_{s_j,t_j} $ are independent.
    \item If $s<t<u$ and $x,z\in\Bbb R$, then we have the a.s. relation 
    \begin{equation}\label{pmod}\tilde Z_{s,u}(x,z) = \int_{\Bbb R} \tilde Z_{s,t}(x,y)\tilde Z_{t,u}(y,z)dy.
    \end{equation}
    \end{enumerate}
\end{prop}

\begin{proof}
    For $x,y\in\Bbb Q$, consider $\phi_n,\psi_n \in C_c^\infty(\Bbb R)$ such that $\phi_n\to \delta_x$ and $\psi_n\to \delta_y$ weakly as $n\to\infty$. Then we have that $(Z_{s,t}, \phi_n\otimes \psi_n)$ form a Cauchy sequence in $L^2(\Omega)$, since the $L^2$-norm of  $(Z_{s,t}, \phi_n\otimes \psi_n)-(Z_{s,t}, \phi_m\otimes \psi_m)$ agrees with the $L^2$-norm of $(U_{s,t}, \phi_n\otimes \psi_n)-(U_{s,t}, \phi_m\otimes \psi_m)$ where $U_{s,t}$ are the propagators for \eqref{she} which is known to be a continuous field satisfying $(U_{s,t}, \phi_n\otimes \psi_n) \to U_{s,t}(x,y)$ a.s. and in $L^2$ as $n\to\infty$.

    By letting $\tilde Z_{s,t}(x,y)$ denote the $L^2$-limit of these random variables for $s,t,x,y\in \Bbb Q$, we obtain a four-parameter field indexed by rationals. By a simple uniform integrability argument and Item (3) in Theorem \ref{mr}, for $n<19 $ ($n\in \Bbb N)$ we obtain that for all $x_j,y_j, s,t\in\Bbb Q$ one has
    \begin{equation}\label{eq}\Bbb E[\prod_{j=1}^n \tilde Z_{s,t}(x_j,y_j)] = \mathbf E[ \prod_{j=1}^n U_{s,t}(x_j,y_j)],
    \end{equation}
    where $U_{s,t}$ are the propagators for \eqref{she}, not necessarily defined on the same probability space. Now by expanding out all powers, this implies that $\Bbb E[ |\tilde Z_{s,t}(x,y)-\tilde Z_{s,t}(x',y)|^{18}]^{1/18} \le C|x-x'|^{1/2}$, also $\Bbb E[ |\tilde Z_{s,t}(x,y)-\tilde Z_{s,t}(x,y')|^{18}]^{1/18} \leq C|y-y'|^{1/2}$, with $C$ uniform over all $(x,x',y,y')$ lying in some compact set of $\Bbb R^4$. 
    These bounds hold because by \eqref{eq} we can replace $\tilde Z_{s,t}$ with $U_{s,t}$ which is known to satisfy these bounds, see e.g. \cite{Wal86} or alternatively \cite[Lemma 3.4 - Lemma 3.9]{A} for a more explicit calculation. 

    Thus by the multiparameter Kolmogorov-Chentsov criterion, for each \textbf{fixed} $(s,t)\in\Bbb Q^2$ there exists a field $\tilde Z_{s,t}(x,y)$ with 
    $x,y\in \Bbb R$ which is a.s. H\"older continuous on compact sets of $\Bbb R^2$. We need to verify that $(Z_{s,t},f) = \int_{\Bbb R^2} \tilde Z_{s,t}(x,y)f(x,y)dxdy$ almost surely for all $s<t$ and all $f\in C_c^\infty(\Bbb R^2)$. To prove this, we have the following general fact: let $\mu$ be a random variable in $\mathcal S'(\mathbb R^2)$ such that for \textit{some} smooth even (deterministic) $\phi \in \mathcal S(\mathbb R^2)$ one has $$\sup_{a,\e}\mathbb E[(\mu,\phi^a_\e)^2]<\infty,\;\;\;\;\;\;\;\;\;\;\;\;\limsup_{\e \to 0} \mathbb E[(\mu,\phi_\e^a)^2]=0, \text{ for all } a\in\mathbb R^2,$$ where for $a\in\Bbb R^2$ and $\e>0$ the rescaled mollifier is given by $\phi_\e^a (x):= \e^{-1}\phi(\e^{-1} (x-a)).$ Then we claim 
    \begin{equation}\label{mu=0}(\mu,\psi)=0 \text{ almost surely for all } \psi \in \mathcal S(\mathbb R^2).
    \end{equation}To prove this, define $\mu^\e(x):= \mu * \phi_\e(x)$ so that $(\mu,\phi^a_\e) = \mu^\e(a).$ Given some smooth $\psi:\mathbb R^2\to \mathbb R$ of compact support note that $(\mu^\e , \psi) \to (\mu,\psi)$ a.s. as $\e \to 0$. This is a purely deterministic statement. Thus it suffices to show that $(\mu^\e,\psi) \to 0$ in probability. To prove that, suppose that the support of $\psi$ is contained in $[-S,S]^2$ and note by Cauchy-Schwarz that $$|(\mu^\e, \psi)| =\bigg|\int_{\mathbb R} \mu^\e(a) \psi(a)da\bigg|\leq \bigg[ \int_{[-S,S]^2} \mu^\e(a)^2 da\bigg]^{1/2} \|\psi\|_{L^2(\mathbb R)}, $$ so that by taking expectation and applying Jensen we find $$\mathbb E[|(\mu^\e, \psi)|] \leq \|\psi\|_{L^2} \bigg[\int_{[-S,S]^2} \mathbb E[\mu^\e(a)^2]da \bigg]^{1/2}.$$ Since by assumption $\mathbb E[\mu^\e(a)^2]$ is bounded in $a$ and $\mathbb E[\mu^\e(a)^2]\to 0$ as $\e\to 0$, the dominated convergence theorem now gives the result by letting $\e \to 0$ on the right side. This proves the claim for $\psi$ of compact support. For general $\psi \in \mathcal S(\mathbb R^2)$ we may find a sequence $\psi_n \to \psi$ in the topology of $\mathcal S(\mathbb R^2)$, with each $\psi_n$ compactly supported. Then $0=(\mu,\psi_n)\to (\mu,\psi)$ a.s. as $n\to\infty$. 

    Finally, to verify that $(Z_{s,t},f) = \int_{\Bbb R^2} \tilde Z_{s,t}(x,y)f(x,y)dxdy$ almost surely for all $s<t$ and all $f\in C_c^\infty(\Bbb R^2)$, apply \eqref{mu=0} with $\mu = Z_{s,t} - \tilde Z_{s,t}$, with $\phi = u\otimes u$ with $u(x) = \pi^{-1/2} e^{-x^2}.$ Simply by construction, we have the two bounds written above. The reason that $\mu$ defines a random element of $\mathcal S'(\Bbb R^2)$ is because of the fact that $\sup_{a,\epsilon} \Bbb E[(\mu,\phi^a_\epsilon)^2]$ is finite for all $\phi\in \mathcal S(\Bbb R^2)$, by Item (3) of Proposition \ref{mr}. See for instance \cite[Theorem 2.7]{CW15} or \cite[Lemma 9]{MWb} for stronger and more general statements.

    Next, we need to verify \eqref{pmod}. To do this, we of course need to use Item (2) in the assumptions of Proposition \ref{mr}. Using the fact that Item (1) of the proposition has already been proved, we can write 
    \begin{align*}\int_{\Bbb R^4 } f(x,z) \epsilon^{-1} \psi&(\epsilon^{-1}(y-w))Z_{t,u}(dy,dz) Z_{s,t}(dx,dw) \\&= \int_{\Bbb R^4} f(x,z) \delta^{-1}\psi(\delta^{-1}(y-w))\tilde Z_{t,u}(y_1,z) \tilde Z_{s,t}(x,y_2)dxdydwdz.
    \end{align*}
    As $\epsilon \to 0$, the left side converges in probability to $\int_{\Bbb R^2} f(x,z) Z_{s,u}(dx,dz)$ which by Item (1) of the proposition is a.s. equal to $\int_{\Bbb R^2} f(x,z) \tilde Z_{s,u}(x,z)dxdz$. By continuity of $\tilde Z$ on compacts, the right side converges almost surely to $\int_{\Bbb R^3} f(x,z) \tilde Z_{t,u}(x,y)\tilde Z_{s,t}(x,y) dxdydz.$ Thus for $s<t<u$ and $f\in C_c^\infty(\Bbb R)$ we have the a.s. relation $$\int_{\Bbb R^2} f(x,z) \tilde Z_{s,u}(x,z)dxdz = \int_{\Bbb R^3} f(x,z) \tilde Z_{t,u}(x,y)\tilde Z_{s,t}(x,y) dxdydz.$$ Now take $f=f_n$ which converges weakly to $\delta_{(x,z)}$ and use the continuity of $\tilde Z$ to obtain \eqref{pmod}.

    So far, this argument shows that for each \textit{fixed} $(s,t)\in \Bbb Q^2$ with $s<t$, we have a continuous function $\tilde Z_{s,t} (x,y)$ such that the desired properties are respected. However, we still need to show the continuity in $(s,t)$ and subsequently extend the field continuously to all four parameters. To do this, we claim the two bounds $\Bbb E[ |\tilde Z_{s,t}(x,y)-\tilde Z_{s,t'}(x,y)|^{18}]^{1/18} \leq C|t-t'|^{1/4}$, and $\Bbb E[ |\tilde Z_{s',t}(x,y)-\tilde Z_{s,t}(x,y)|^{18}]^{1/18} \leq C|s'-s|^{1/4}.$ To see this, one expands everything out and uses \eqref{eq} in conjunction with the propagator equation \eqref{pmod}, then notes that the corresponding bounds do hold for $U_{s,t}$, see \cite{Wal86} or alternatively \cite[Lemma 3.4 - Lemma 3.9]{A} for a more explicit calculation. By multiparameter Kolmogorov-Chentsov, the required modification then exists on the same probability space, and this is exactly the place where we need enough moments (18 to be exact).
\end{proof}

\begin{prop}[Solving the martingale problem] \label{2}
    Let $\tilde Z_{s,t}(x,y)$ be the continuous modification from Proposition \ref{mr0}. Fix some $x,s\in \Bbb R$. Then for all $\phi\in C_c^\infty(\Bbb R)$ the process $$M_t(\phi) := (\tilde Z_{s,t}(x,\bullet),\phi)_{L^2(\Bbb R)}-\frac12 \int_s^t (\tilde Z_{s,u}(x,\bullet),\phi'')_{L^2(\Bbb R)}du,\;\;\;\;\;\;\;\;\;\;\; t\ge s$$ is a continuous martingale in the filtration $(\mathcal F_{s,t})_{t\ge s}$ where $\mathcal F_{s,t}:=\sigma(\tilde Z_{a,b}(x,y): s<a<b<t,x,y\in\Bbb R\}$. The quadratic variations are furthermore given by $$\langle M(\phi)\rangle_t = \int_s^t (\tilde Z_{s,u}(x,\bullet)^2 , \phi^2)_{L^2(\Bbb R)} du.$$
\end{prop}

\begin{proof}
    Without loss of generality we can assume $s=x=0$. Let $z_t(y):= Z_{0,t}(0,y).$ Applying \eqref{pmod} and using independence of $\tilde Z_{0,s}$ and $\tilde Z_{s,t}$, one obtains that $$\Bbb E[ z_t (y) | \mathcal F_{0,s}] = p_{t-s} * z_s(y).$$ where $p_t(x):= (2\pi t)^{-1/2} e^{-x^2/2t}$ is the standard heat kernel. This means that if $s<t$ then $\Bbb E[ (z_t,\phi)_{L^2(\Bbb R)}|\mathcal F_s] = (z_s, p_{t-s}*\phi)_{L^2(\Bbb R)}$, and from here the martingality of $M_t(\phi)$ is an easy and direct consequence. Now we need to compute the quadratic variations. For this, we claim that for any deterministic sequence of partitions of $[0,t]$ with mesh tending to zero, one has $$\lim_{N\to \infty} \Bbb E\bigg[ \bigg( \sum_{i=1}^N (z_{t_{i+1}}-z_{t_i}, \phi)_{L^2(\Bbb R)}^2 - \sum_{i=1}^N (z_{t_i}^2 ,\phi^2)_{L^2(\Bbb R)} (t_{i+1}-t_i) \bigg)^2\bigg] =0. $$ Indeed if $(t,x)\mapsto \mathcal U_t(x)$ solves \eqref{she} with Dirac initial condition, then notice from the assumptions of Proposition \ref{mr} that 
    \begin{align*}\Bbb E\bigg[ \bigg( &\sum_{i=1}^N (z_{t_{i+1}}-z_{t_i}, \phi)_{L^2(\Bbb R)}^2 - \sum_{i=1}^N (z_{t_i}^2 ,\phi^2)_{L^2(\Bbb R)} (t_{i+1}-t_i) \bigg)^2\bigg]\\&= \mathbf E\bigg[ \bigg( \sum_{i=1}^N (\mathcal U_{t_{i+1}}-\mathcal U_{t_i}, \phi)_{L^2(\Bbb R)}^2 - \sum_{i=1}^N (\mathcal U_{t_i}^2 ,\phi^2)_{L^2(\Bbb R)} (t_{i+1}-t_i) \bigg)^2\bigg].
    \end{align*}
    This follows from a calculation, expanding everything out algebraically and then using the propagator equation \eqref{pmod} together with the moment equality \eqref{eq} repeatedly, to show that replacing $z_t$ by $\mathcal U_t$ does not change the above expectation. But the desired limit is indeed true for the process $\mathcal U_t$ since $\mathcal U_t$ does solve the martingale problem for \eqref{she}, and because the associated quadratic variation of the martingale for $\mathcal U$ is known to be given by $\int_0^t (\mathcal U_s^2, \phi^2)_{L^2(\Bbb R)}ds,$ see e.g. \cite[Proposition 4.11]{BG97}. The claim thus follows immediately.
\end{proof}

\begin{cor}\label{solves}
    Let $\tilde Z_{s,t}(x,y)$ be as in Proposition \ref{mr0}. Then for all fixed $s,x\in\Bbb R$, the law of the process $(t,y) \mapsto \tilde Z_{s,t}(x,y)$ agrees with that of the solution to \eqref{she} started at time $s$ and from initial condition $\delta_x(y)$.
\end{cor}

\begin{proof}
    Again we can assume $s=x=0$, so let $z_t(y):= Z_{0,t}(0,y).$ The martingale problem from Proposition \ref{2}  identifies the law of the process $(t,y)\mapsto z_{t+s}(y)$ as that of the solution of the multiplicative-noise heat equation with initial condition $z_s$, for $s>0$ \cite[Proposition 4.11]{BG97}. However, it remains to identify the initial data as a Dirac mass, i.e., take the $s\downarrow 0$ limit rigorously. In \cite[Section 6]{Par19} there is a general approach to do exactly this. Specifically, it suffices to show as in Lemma 6.6 of and Equation (64) of that paper the two bounds 
		\begin{align}\label{delta} & \Ex[z_t(x)^2]\leq \Con \cdot t^{-1/2}K(t,x),\\ \label{delta1}
			& \Ex[ |z_t(x)-K(t,x)|^2] \leq \Con \cdot K(t,x),
		\end{align} where $K$ is the standard heat kernel, and $\Con$ is independent of $t>0$ and $x\in\mathbb R$. The solution of \eqref{she} with $\delta_0$ initial condition certainly satisfies this bound, and by the assumptions in Proposition \ref{mr}, we know $\Ex[(\int_\mathbb R z_t(x)\phi(x)dx)^k] = \mE[(\int_\mathbb R \mathcal{U}_t(x)\phi(x)dx)^k]$ for  $k=1,2,3,4$ and $\phi\in C_c^\infty(\mathbb R)$, where $(t,x) \mapsto \mathcal{U}_t(x)$ solves \eqref{she} with $\delta_0$ initial data. From here we can conclude by letting $\phi \to \delta_x$ that $\Ex[z_t(x)^k] = \mE[ \mathcal{U}_t(x)^k]$ for all $k\in \{1,2\}$ and all $x\in\mathbb R$.
\end{proof}

\begin{prop}[Identifying the joint noise and defining it on the same probability space] \label{jnoise}
    Let $\tilde Z_{s,t}$ be as in Proposition \ref{mr0}. On can construct a space-time white noise $\xi$ on the same probability space, such that if $s,x\in\Bbb R$, then $(t,y)\mapsto \tilde Z_{s,t}(x,y)$ almost surely solves \eqref{she} driven by that same noise $\xi$. Thus we have completed Proposition \ref{mr}.
\end{prop}

\begin{proof}
    Throughout this proof, fix $s\in\Bbb R$, and let $\{x_i\}_{i\ge 1}$ be some enumeration of $\Bbb Q$. 
    Let $z_i(t,y):= \tilde Z_{s,t}(x_i,y).$ 
    We will argue that the driving noise $\xi_i$ of the process $z_i$ is an unambiguous notion, and furthermore that all $\xi_i$ are the same noise, i.e., $\xi_i=\xi_1.$

    Assume without loss of generality that $s=0$. By using the result of Corollary \ref{solves}, we now argue that each $\xi_i$ can be recovered measurably from the process $z_i$ itself. This will be a consequence of the Walsh theory of stochastic integrals \cite{Wal86}. More precisely, we let $M^{i}_t(\phi):= (z_i(t,\bullet),\phi)_{L^2(\Bbb R)} -\frac12 \int_0^t (z_i(s,\bullet),\phi'')_{L^2(\Bbb R)}ds$ denote the martingales from above but corresponding to the respective initial data $\delta_{x_i}$. These are \textit{orthomartingales} in the sense of \cite{Wal86}, thus we can define stochastic integrals against them. 
    Then as observed in \cite{KS88}, Walsh's theory guarantees that we can measurably recover the driving noises by the stochastic integral formula $$(\xi_i, \phi\otimes 1_{[0,t]})_{L^2(\Bbb R_+\times \Bbb R)} = \int_0^t \int_\Bbb R z_i(s,y)^{-1}\phi(y) M^i(ds,dy). $$ This formula implies that $(\xi_i, \phi\otimes 1_{[0,t]})$ is a martingale in $t$ with quadratic variation $\|\phi\|_{L^2(\Bbb R)}^2 t$, which implies that each $\xi_i$ is a Gaussian space-time white noise (defined on the same probability space) by Levy's criterion. A similar calculation as above in Proposition \ref{2} yields $\langle M^{i}(\phi), M^{j}(\phi)\rangle_t = \int_0^t (z^i_sz_s^j,\phi^2 ) ds  $. It follows that $\xi_i= \xi_j$ for the respective noises, since the above formula yields the relation $\Bbb E[ (\xi_i,\phi\otimes 1_{[0,t]})(\xi_j,\phi\otimes 1_{[0,t]})] = t \|\phi\|_{L^2(\Bbb R)}^2$ which in turn yields $\Bbb E[(\xi_i-\xi_j,\phi\otimes 1_{[0,t]})^2]=0.$

    From $\Bbb Q$ one can extend the result to $\Bbb R$ simply because $\tilde Z$ is already known to be continuous in all four variables. Now if $s_1<s_2$ and $x_1,...,x_n\in \Bbb R$ and one wishes to show that the driving noise of $(t,y) \mapsto Z_{s_i,t}(x_j,y)$ are the same for all $i,j$, note that by the Markov property that is guaranteed by Proposition \ref{solves}, one can restart the process at time $s_2$, which reduces us back to the case where $s_1=s_2$ which has already been covered above.
\end{proof}

\begin{rk}\label{wm} Notice that we only used Item (3) of Proposition \ref{mr} up to $n=19$ rather than all $n\in \Bbb N$. This is because we used an $18^{th}$ moment bound in the proof of Proposition \ref{mr0}. This $19^{th}$ moment assumption could be further relaxed, but the proof would become longer. 

Let us illustrate how to weaken the moment assumptions in the simpler context of the ``Geometric Brownian flow" of Proposition \ref{gbf}. We claim that in Item (3) of that proposition, the assumption of 6 moments can be weakened to 4 moments. To prove it, note that $\tilde X_t:= e^{-\frac{t}2} G_{0,t}$ is a martingale indexed by $\Bbb Q$, thus it admits a cadlag extension to $\Bbb R_+$ by the fundamental regularization lemma, see \cite[Chapter II, Theorem (2.5)]{RY99}. Then \eqref{qvlim} still implies that the (optional) quadratic variation of the martingale $\tilde X$ is given by $[\tilde X]_t = \int_0^t \tilde X_s^2 ds.$ Continuity of $[\tilde X]$ implies continuity of $\tilde X$ itself, thanks to the identity $\delta [Y] = (\delta Y)^2$ for any semimartingale $Y$, where $(\delta Y)_s:= Y_s-Y_{s-}.$ Thus $\tilde X_t$ and $\tilde X_t^2 - \int_0^t \tilde X_s^2 ds$ are both continuous martingales, completing the proof exactly as before. But Item (3) has only been used up to $n=4$ rather than $n=6$, since we avoided the use of Kolmogorov-Chentsov to construct the continuous modification.

This illustrates the fact that that when we assume less moments, we can no longer keep the proof within the realm of \textit{continuous} processes, thus needing to think about how to deal with jumps or other irregularities. Very similar considerations apply in the SPDE case of Proposition \ref{mr}: we can assume less than 19 moments, but in that case the proof will become so much longer that we have not pursued it in this paper. The reason is that it would involve introducing various other topologies that are weaker than that of $C(\Bbb R^4_\uparrow)$ and harder to define, then showing retroactively that the resultant four-parameter process $\tilde Z$ ends up being an element of $C(\Bbb R^4_\uparrow)$ anyways. Alternatively, one can use Tsai's exchange method to avoid this \cite{Ts24}. The version with all moments will suffice for us.
\end{rk}

\section{Proof of Theorem \ref{mr2}}\label{s3}



Henceforth we consider the propagators $Z^\epsilon_{s,t}(x,y)$ of the equation given in Theorem \ref{mr2}. Already, these almost surely satisfy the propagator equation 
\begin{equation}\label{propa}Z^\epsilon_{s,u} (x,z) = \int_{\Bbb R} Z_{s,t}^\epsilon(x,y)Z_{t,u}^\epsilon(y,z)dy,\;\;\;\;\;\;\;\;\;\;\;\; s<t<u, \;\;\;\;\;\;\;\;\;\;x,z\in\Bbb R.
\end{equation}
To prove this, one can show that both sides of the equation satisfy the Duhamel form of the equation written in Theorem \ref{mr2}, for which \textit{pathwise uniqueness} is well-known \cite{Wal86} for any fixed $\epsilon>0$. We do not write the calculation here, but the details of a similar calculation can be found in e.g. \cite[Theorem 3.1(vii)]{AKQ14b} or \cite[Lemma 3.12]{A}. 

The bulk of this section will be devoted to calculating the limit of the moments as $\epsilon\to 0$ of the processes $Z^\epsilon_{s,t}$, since that is what is needed to apply Proposition \ref{mr}. 

\begin{prop}[Moment formula] \label{3.1} Let $\Phi_\epsilon := \varphi_\epsilon * \varphi_\epsilon$, with $\varphi_\epsilon$ as in Theorem \ref{mr2}. If $\phi: \Bbb R\to\Bbb R$ is bounded and continuous, and $x_1,...,x_n\in\Bbb R$, we have that $$\int_{\Bbb R^n} \Bbb E [ \prod_{j=1}^n Z_{0,t}^\epsilon (x_j, y_j) ]\prod_{j=1}^n \phi(y_j) dy\cdots dy_n = \mathbf E^{(x_1,...,x_n)} _{\mathrm{BM}^{\otimes n}} [ e^{ \epsilon^{2p-\frac12} \sum_{1\le i<j\le n}\int_0^t \Phi_\epsilon^{(2p)} (W^i_s-W^j_s)ds} \prod_{j=1}^n \phi(W^j_t)].$$Here the expectation on the right side is with respect to a $n$-dimensional Brownian motion $W$ started from $(x_1,...,x_n)$, and $\Phi^{(k)}$ denotes the $k^{th}$ derivative of $\Phi$.
\end{prop}

The proof follows from using the Feynman-Kac representation of $Z^\epsilon$, see \cite[Section 2]{BC95} for an identical calculation (neither the factor $\epsilon^{2p-\frac12}$ nor the $(2p)^{th}$ derivative on $\Phi_\epsilon$ appear there, since their driving noise is just $\eta^\epsilon$ as opposed to $\epsilon^{p-\frac14} \partial_y^p\eta^\epsilon$, but this simply corresponds to replacing the mollifier $\varphi_\epsilon$ by $\epsilon^{p-\frac14} \partial_y^p\varphi_\epsilon$, which precisely leads to the replacement $\Phi_\epsilon\to \epsilon^{2p-\frac12} \Phi_\epsilon^{(2p)}$ in the above moment formula). By It\^o's formula, note that we can rewrite this as 
\begin{align}
    \notag \int_{\Bbb R^n } \Bbb E & \bigg[ \prod_{j=1}^n Z_{0,t}^\epsilon (x_j, y_j) \bigg] \prod_{j=1}^n \phi(y_j) dy\cdots dy_n \\&= \notag \mathbf E_{\mathrm{BM}^{\otimes n}}^{(x_1,...,x_n)} \bigg[ e^{\epsilon^{2p-\frac12} \sum_{1\le i<j\le n} \Phi_\epsilon^{(2p-2)}(W^i_t-W^j_t) -\int_0^t \Phi_\epsilon^{(2p-1)}(W^i_s-W^j_s) d(W^i-W^j)_s } \prod_{j=1}^n \phi( W^j_t )\bigg] \\&= \mathbf E^{(x_1,...,x_n)}_{\mathrm{Diff}(\epsilon)} \bigg[e^{ \mathbf T^\epsilon_t(X)+\sum_{1\le i<j\le n} \epsilon^{2p-\frac12}\Phi_\epsilon^{(2p-2)}(X^i_t-X^j_t) +\epsilon^{4p-1} \int_0^t \Phi_\epsilon^{(2p-1)}(X^i_s-X^j_s)^2 ds } \prod_{j=1}^n \phi( X^j_t) \bigg] \label{apx}
\end{align}
where by Girsanov's theorem the latter expectation is with respect to the \textbf{diffusion process} on $\Bbb R^n$, started from initial condition $(x_1,...,x_n)$, with dynamics given by $$dX^i = \epsilon^{2p-\frac12} \sum_{j\ne i} \mathrm{sign}(i-j) \Phi_\epsilon^{(2p-1)}(X^i-X^j)ds + dW^i,$$ and where 
\begin{align*}\mathbf T_t^\epsilon(X)&:= \frac12 \epsilon^{4p-1}\sum_{i=1}^n \sum_{\substack{j_1,j_2\neq i\\ j_1\ne j_2}} \int_0^t \mathrm{sign}(i-j_1)\mathrm{sign}(i-j_2) \prod_{k=1,2} \Phi^{(2p-1)}_\epsilon(X^i_s -X^{j_k}_s) ds.
\end{align*}
\begin{defn}
    We say that a family of functions $u_\epsilon: \Bbb R\to \Bbb [0,\infty)$ with $0<\epsilon\le 1$ is an approximate Dirac mass if $u_\epsilon(y) = \delta^{-1}\psi(\delta^{-1}y)$ for some continuous function $\psi$ of compect support with $\int_\Bbb R\psi=1$.
\end{defn}

Notice that the function $\epsilon^{4p-1} (\Phi_\epsilon^{(2p-1)})^2$ appears in the exponential in \eqref{apx}. 
This family of functions is an approximate Dirac mass centered at the origin, multiplied by the deterministic constant $\sigma_p^2 := \int_\Bbb R (\varphi * \varphi^{(2p-1)}(x))^2dx.$ Meanwhile the process $\mathbf T^\epsilon_t(X)$ will be shown to be irrelevant in the limit, which is because it only depends on ``triple intersections" between the different coordinates $X^1,...,X^n$. These are the the crucial observations in proving Theorem \ref{mr2}.

\begin{lem}\label{CAF}
    Let $\mathbf E_{\mathrm{BM}}^x$ denote expectation with respect to a standard Brownian motion on $\Bbb R$, started from $x\in\Bbb R$. There exists some universal constant $C>0$, such that uniformly over all smooth functions $\phi\ge 0$ with $\int_{\Bbb R} \phi =1,$ and all $s,t ,q\ge 0$, one has 
    \begin{align*}
        \sup_{x\in\Bbb R} \mathbf E^x_{\mathrm{BM}}\bigg[\bigg|\int_s^t \phi(B_s)ds\bigg|^q\bigg]^{1/q} &\leq Cq^{1/2} |t-s|^{1/2}\\
        \sup_{x\in\Bbb R} \mathbf E^x_{\mathrm{BM}}[e^{q \int_0^t \phi(B_s)ds}] &\leq Ce^{Cq^2 t}.
    \end{align*}
\end{lem}
\begin{proof}
    First note the trivial bound $\sup_{x\in \Bbb R} \mathbf E^x_{\mathrm{BM}}[\phi(B_t)] \leq c\|\phi\|_{L^1(\Bbb R)} t^{-1/2},$ where $c= 1/\sqrt{2\pi}$ is universal. Thus if $\|\phi\|_{L^1(\Bbb R)}=1$ and $m\in\Bbb N$, we can iteratively apply the Markov property of Brownian motion $m$ times to see that 
    \begin{align*}\mathbf E^x_{\mathrm{BM}}\bigg[\bigg( \int_s^t \phi(B_s)ds\bigg)^m \bigg] &= m! \int_{s\le s_1\le s_2...\le s_m\le t} \Bbb E \bigg[ \prod_{j=1}^m \phi(B_{s_j}) \bigg] ds_1\cdots ds_m \\ &\leq m! c^m \int_{s\le s_1\le s_2...\le s_m\le t}\prod_{j=0}^{m-1} (s_{i+1}-s_i)^{-1/2} ds_1\cdots ds_m \\&\leq C^m (t-s)^{m/2} \Gamma(m/2), 
    \end{align*}
    where $\Gamma$ is the Gamma function, and $C\ge c$ is some larger but still universal constant. From here, the first bound is immediate. For the second bound, take $s=0$, then divide both sides by $m!$ and sum over all $m$ to obtain the required bound. 
\end{proof}

\begin{prop}[Limit of the moments]\label{lim}
    Let $\phi: \Bbb R\to \Bbb R$ be bounded and continuous. For all $n\in\Bbb N,$ and all $(x_1,...,x_n)\in\Bbb R^n$, we have $$\lim_{\epsilon\to 0} \int_{\Bbb R^n } \Bbb E \bigg[ \prod_{j=1}^n Z_{0,t}^\epsilon (x_j, y_j) \bigg] \prod_{j=1}^n \phi(y_j) dy\cdots dy_n = \mathbf E_{\mathrm{BM}^{\otimes n}}^{(x_1,...,x_n)} \bigg[ e^{\sigma_p^2  \sum_{1\le i<j\le n} L_0^{W^i-W^j}(t) } \prod_{j=1}^n \phi(W^j_t)\bigg] .$$
    Here $\sigma_p^2 = \int_\Bbb R (\varphi * \varphi^{(2p-1)}(x))^2dx,$ with $\varphi$ as in Theorem \ref{mr2}. The above convergence is uniform on compact subsets of $(x_1,...,x_n)$.
\end{prop}

\begin{proof}
    Using the diffusion representation explained after Proposition \ref{3.1} above, we just need to calculate the limit as $\epsilon\to 0$ of the expectation \begin{equation}\label{limit}\lim_{\epsilon\to 0} \mathbf E^{(x_1,...,x_n)}_{\mathrm{Diff}(\epsilon)} \bigg[e^{ \mathbf T_t^\epsilon(X)+\sum_{1\le i<j\le n} \epsilon^{2p-\frac12}\Phi_\epsilon^{(2p-2)}(X^i_t-X^j_t) +\epsilon^{4p-1} \int_0^t \Phi_\epsilon^{(2p-1)}(X^i_s-X^j_s)^2 ds } \prod_{j=1}^n \phi(X^j_t) \bigg].\end{equation}
    On one hand, one has a deterministic bound $\epsilon^{2p-\frac12}|\Phi_\epsilon^{(2p-2)}|\leq \epsilon^{1/2}$, thus the first term in the exponential $\epsilon^{2p-\frac12}\Phi_\epsilon^{(2p-2)}(X^i_t-X^j_t)$ is completely inconsequential in the limit, and will be ignored henceforth. To take the limit of the remaining parts, it is clear that we require a joint invariance principle for the entire collection of processes given by $$\bigg( \big(X^1_t,...,X^n_t\big) , \;\;\big( \epsilon^{4p-1}  \int_0^t \Phi_\epsilon^{(2p-1)}(X^i_s-X^j_s)^2 ds\big)_{1 \le i<j\le n},\;\; \mathbf T_t(X)\bigg)_{t\in [0,T]},$$
    under the measures $\mathbf E^{(x_1,...,x_n)}_{\mathrm{Diff}(\epsilon)}$ as $\epsilon\to 0$. Specifically, we will now show that these processes have a \textbf{joint} limit in distribution given by $$\bigg( \big( W^1_t,..., W^n_t\big), \big( \sigma_p^2 L_0^{W^i-W^j}(t)\big)_{1\le i<j\le n},\;\;\;0\bigg)_{t\in [0,T]} ,$$ where $(W^1,...,W^n)$ is a standard Wiener process started from $(x_1,...,x_n)$ and $L^X_0$ is the local time process of the semimartingale $X$. The convergence in distribution is with respect to the topology of $C[0,T]$ in each coordinate (that is, uniform convergence). To prove this, we will first focus on  $(X^1,...,X^n)$ before studying the remaining parts.
    
    For any continuous martingale $G$, if $\Bbb E [ e^{ 2\langle G\rangle_t} ]<\infty$ for all $t\ge 0$, then by Novikov's condition $e^{ 2G_t-2\langle G\rangle_t} $ is a martingale, thus we have the bound $$\Bbb E[e^{G_t}] = \Bbb E[ e^{G_t- \langle G\rangle_t +\langle G \rangle_t}] \leq \Bbb E[e^{2G_t -2\langle G\rangle_t }]^{1/2} \Bbb E [ e^{ 2\langle G\rangle_t} ]^{1/2} = \Bbb E [ e^{ 2\langle G\rangle_t} ]^{1/2}.$$
    Thus if we let $\mathbf E_{\mathrm{BM}}^x$ denote expectation with respect to a standard Brownian motion $B$ on $\Bbb R$, started from $x\in\Bbb R$, then for any $q\ge 1$ one has by Lemma \ref{CAF} that $$\sup_{\substack{\epsilon\in (0,1]\\x\in\Bbb R}} \mathbf E_{\mathrm{BM}}^x \big[ e^{ q \epsilon^{2p-\frac12} \int_0^T \Phi_\epsilon^{(2p-1)}(B_s)dB_s} \big] \leq \sup_{\substack{\epsilon\in (0,1]\\x\in\Bbb R}} \mathbf E_{\mathrm{BM}}^x[ e^{2q^2\epsilon^{4p-1} \int_0^T \Phi_\epsilon^{(2p-1)}(B_s)^2 ds} \big]^{1/2} \le Ce^{\sigma_p^2q^4 CT} <\infty,$$ where $C$ is some universal constant. By H\"older's inequality and Girsanov's theorem, the last expression easily implies that if $F: C[0,T]\to\Bbb R$ is any measurable path functional, then \begin{equation}
        \label{sis} \mathbf E^{(x_1,...,x_n)}_{\mathrm{Diff}(\epsilon)} [ F(X)] \le C \mathbf E_{\mathrm{BM}^{\otimes n}}^{(x_1,...,x_n)} [ F(W) ^{2}]^{1/2}
    \end{equation}where $C$ is an \textit{absolute constant} independent of $\epsilon\in (0,1]$ and $F$ (specifically $C$ is an upper bound, uniform over all $\epsilon$, for the $L^2$ norm of the Radon-Nikodym derivative of $X$ with respect to $W$). This will be a very useful fact going forward. 
    
    Letting $F_\alpha(f)$ denote the $\alpha$-Holder norm of $f$ on $[0,T]$, where $\alpha<1/2$, note by \eqref{sis} and Fernique's theorem that we immediately obtain tightness of the processes $(X^1,...,X^n)$ in $C[0,T]$, under $\mathbf E^{(x_1,...,x_n)}_{\mathrm{Diff}(\epsilon)}$ as $\epsilon\to 0$. Now we wish to identify the limit point as Brownian motion, for which we will use Levy's criterion. By Dynkin's formula, note that 
    \begin{align*}M^i_t&:= X^i_t - \epsilon^{2p-\frac12} \sum_{j\ne i} \mathrm{sign}(i-j)\int_0^t \Phi_\epsilon^{(2p-1)}(X^i_s-X^j_s)ds \\ Q^{ij}_t& = M^i_tM^j_t -\delta_{ij} t
    \end{align*}
    are both continuous $\mathbf E^{(x_1,...,x_n)}_{\mathrm{Diff}(\epsilon)}$-martingales, for $\epsilon>0$. Indeed, martingality of $M^i$ follows from considering the first-order part of the generators, while martingality of $Q^{ij}$ comes from looking at the second-order part which is just the Laplacian.
    
    Using that $\epsilon^{2p-1} |\Phi_\epsilon^{(2p-1)}|$ is an approximate Dirac mass, Lemma \ref{CAF} yields $$\mathbf E_{\mathrm{BM}^{\otimes n}}^{(x_1,...,x_n)} \big[ \big| \int_s^t \epsilon^{2p-1}| \Phi_\epsilon^{(2p-1)}(W^i_u-W^j_u)|du\big|^q]^{1/q} \leq C_q(t-s)^{1/2},$$ where $C_q$ does not depend on $\epsilon$. Thus by \eqref{sis}, we obtain that \begin{equation}\label{addfct}\epsilon^{2p-\frac12} \mathbf E^{(x_1,...,x_n)}_{\mathrm{Diff}(\epsilon)} \bigg[  \bigg| \int_s^t |\Phi_\epsilon^{(2p-1)}(X^i_u-X^j_u)|du\bigg|^q\bigg]^{1/q} \leq C_q \epsilon^{1/2} |t-s|^{1/2}\end{equation}
    which implies that under $\mathbf E^{(x_1,...,x_n)}_{\mathrm{Diff}(\epsilon)},$ the processes $\epsilon^{2p-\frac12} \int_0^t \Phi_\epsilon^{(2p-1)}(X^i_s-X^j_s)ds$ tend to 0 in law as $\epsilon\to 0$, in the topology of $C[0,T]$. Thus for any limit point $(W^1,...,W^n)$ of the processes $(X^1,...,X^n)$ viewed under $\mathbf E^{(x_1,...,x_n)}_{\mathrm{Diff}(\epsilon)}$ as $\epsilon\to 0$, the processes $W^i_t$ and $W^i_tW^j_t-\delta_{i,j}t$ are both continuous martingales. This is because martingality is preserved by limit points as long as uniform integrability holds, and uniform integrability \textit{does} hold thanks to \eqref{sis}. By Levy's criterion, this identifies the limit point $(W^1,...,W^n)$ as a standard Brownian motion in $\Bbb R^n$. 

    It remains to study the processes $\big( \epsilon^{4p-1}  \int_0^t \Phi_\epsilon^{(2p-1)}(X^i_s-X^j_s)^2 ds\big)_{1 \le i<j\le n,\; 0\le t \le T}$ as $\epsilon\to 0$, if we take a joint limit in distribution with $(X^1,...,X^n)$ under the measures $\mathbf E^{(x_1,...,x_n)}_{\mathrm{Diff}(\epsilon)}$. Define the functions $\Psi_\epsilon:\Bbb R\to \Bbb R$ to be the unique solution of the second-order ODE given by $$\Psi_\epsilon '' = \epsilon^{4p-1} (\Phi_\epsilon^{(2p-1)})^2 
    ,\;\;\;\;\;\;\;\;\; \lim_{y\to-\infty} \Psi_{\epsilon}(y)=0. $$
    Using the fact that $\epsilon^{4p-1} (\Phi_\epsilon^{(2p-1)})^2$ is $\sigma_p^2$ times an approximate Dirac mass, 
    it is clear that $\Psi_\epsilon$ converges uniformly on compacts to the function $x\mapsto 2\sigma_p^2 \max\{0,x\},$ as $\epsilon\to 0.$ On the other hand, letting $\Lambda^{ij}_\epsilon(x_1,...,x_n):= \epsilon^{2p-1} \sum_{1\le k<\ell \le n} \Phi_\epsilon^{(2p-1)}(x_k-x_\ell)\cdot \big( \partial_k-\partial_\ell\big) \big(\Psi_\epsilon(x_i-x_j)\big)$, Dynkin's formula implies that the processes $$D^{ij}_t := \Psi_\epsilon(X_t^i-X_t^j) - \epsilon^{4p-1} \int_0^t \Phi_\epsilon^{(2p-1)}(X^i_s-X^j_s)^2 ds - \epsilon^{\frac12} \int_0^t \Lambda^{ij}_\epsilon(X^1_s,...,X^n_s)ds.$$ are martingales for $1\le i<j\le n$. Under the measures $\mathbf E^{(x_1,...,x_n)}_{\mathrm{Diff}(\epsilon)}$, the middle process on the right side given by $\epsilon^{4p-1} \int_0^\bullet \Phi_\epsilon^{(2p-1)}(X^i_s-X^j_s)^2 ds$ may be shown to be tight in $C[0,T]$ with uniform-in-$\epsilon$ $L^p$ bounds, using the exact same types of arguments as we used to prove \eqref{addfct}, first proving the bound for Brownian motion and then leveraging \eqref{sis}. Likewise since $|\Psi_\epsilon'|$ is deterministically bounded independently of $\epsilon\in(0,1]$, and since $\epsilon^{2p-1}|\Phi_\epsilon^{(2p-1)}|$ is an approximate Dirac mass, the processes $\int_0^\bullet \Lambda^{ij}_\epsilon(\vec X_s)ds$ are tight in $C[0,T]$ by the same argument, and therefore the last term on the right side converges to zero in the topology of $C[0,T]$ thanks to the multiplying factor $\epsilon^{1/2}$.
    
    Let $(W^1,...,W^n, (\mathscr V^{ij})_{1\le i<j\le n})$ be a joint limit point of the entire tuple of $\frac12n(n+1)$ processes given by $\big( (X^1_t,...,X^n_t) , ( \epsilon^{4p-1}  \int_0^t \Phi_\epsilon^{(2p-1)}(X^i_s-X^j_s)^2 ds\big)_{1 \le i<j\le n}\big)_{t\in [0,T]}$, viewed under $\mathbf E^{(x_1,...,x_n)}_{\mathrm{Diff}(\epsilon)}$ as $\epsilon\to 0$. We already proved that the marginal law of $(W^1,...,W^n)$ must be a Brownian motion. Since martingality is preserved by limit points as long as uniform integrability holds, the martingality of $D^{ij}$ in the prelimit implies that $2\sigma_p^2 \max\{0, W^i_t-W^j_t\} - \mathscr V^{ij}_t$ is a continuous martingale for the limit point. By Tanaka's formula, this forces $\mathscr V^{ij}_t= \sigma_p^2 L_0^{W^i-W^j}(t)$, for \textbf{all} $1\le i<j\le n,$ thus uniquely identifying the limit point of the full tuple of $\frac12n(n+1)$ processes.

    Now we need to deal with the process $\mathbf T_t(X)$, and show that it converges to 0 in the topology of $C[0,T]$, viewed under $\mathbf E^{(x_1,...,x_n)}_{\mathrm{Diff}(\epsilon)}$ as $\epsilon\to 0$. Note that the increments of $\mathbf T_t(X)$ are deterministically upper-bounded by the increments of $\mathbf S_t(X)$ where $$\mathbf S_t^\epsilon(X):= \epsilon^{4p-1}\sum_{i=1}^n \sum_{\substack{j_1,j_2\neq i\\ j_1\ne j_2}} \int_0^t \prod_{k=1,2} |\Phi^{(2p-1)}_\epsilon(X^i_s -X^{j_k}_s)| ds. $$
    Use the bound $$\prod_{k=1,2} |\Phi^{(2p-1)}_\epsilon(X^i_s -X^{j_k}_s)|\leq \sum_{k=1,2} \Phi^{(2p-1)}_\epsilon(X^i_s -X^{j_k}_s)^2$$
    and recall that we already showed above that processes given by $\epsilon^{4p-1} \int_0^\bullet \Phi_\epsilon^{(2p-1)}(X^i_s-X^j_s)^2 ds$ are tight in $C[0,T]$ with uniform-in-$\epsilon$ $L^p$ bounds. Thus we see that $\mathcal S_t^\epsilon(X)$ is also tight in $C[0,T]$ with uniform-in-$\epsilon$ $L^p$ bounds. It remains to identify the limit point as 0. Let $(W^1,...,W^n,S)$ be a joint limit point as $\epsilon\to 0$ of $(X^1,...,X^n,\mathbf S^\epsilon(X))$ then notice that $S$ is a continuous nondecreasing process that is constant on every open interval in the complement of the closed set where at least three of the coordinates $(W^1,...,W^n)$ are equal. Indeed $\mathbf S^\epsilon(X)$ can only increase on the set of $t$ where there exist three distinct indices $(i,j_1,j_2)$ such that $| X^i - X^{j_k} | < \epsilon$ for \textit{both} values of $k$; thus the limit point $S$ can only increase where there exist $ (i,j_1,j_2)$ such that $W^i = W^{j_k}$ for both $k$, thus three separate coordinates of the process $W$ must be equal for $S$ to be able to increase. But the marginal law of $(W^1,...,W^n)$ has been shown to be a Brownian motion, and it is known that a three-dimensional Brownian motion never returns to the origin, thus $S$ must be a \textit{constant} process. Thus $S=0,$ as desired.

    Finally, we take the limit in \eqref{limit}. Using the invariance principle and the uniform integrability guaranteed by \eqref{sis}, one has \begin{align*}\lim_{\epsilon\to 0} \mathbf E^{(x_1,...,x_n)}_{\mathrm{Diff}(\epsilon)} \bigg[&e^{ \mathbf T_t(X)+\sum_{1\le i<j\le n} \epsilon^{4p-1} \int_0^t \Phi_\epsilon^{(2p-1)}(X^i_s-X^j_s)^2 ds } \prod_{j=1}^n \phi(X^j_t) \bigg] \\&= \mathbf E_{\mathrm{BM}^{\otimes n}}^{(x_1,...,x_n)} \bigg[ e^{\sigma_p^2  \sum_{1\le i<j\le n} L_0^{W^i-W^j}(t) } \prod_{j=1}^n \phi(W^j_t)\bigg].
    \end{align*}
    The convergence is uniform on compacts because allowing $x_i=x_i^\epsilon$ to vary within some bounded set as $\epsilon\to 0$ will not affect the invariance principle or the uniform integrability.
\end{proof}

\begin{proof}[Proof of Theorem \ref{mr2}]

We wish to show convergence in law with respect to finite-dimensional laws in the (very weak) topology of $\mathcal M(\Bbb R^2)$. One may show that $\mathcal M(\Bbb R^2)$ with its vague topology is indeed a Polish space, so that tightness of some family of probability measures does imply existence of subsequential limit points. The tightness is actually immediate from Banach-Alaoglu and the fact that $\Bbb E[ (Z^\epsilon_{s,t},f)^n]$ remains bounded as $\epsilon \to 0$, for all $f\in C_c^0(\Bbb R^2)$ and $n\in\Bbb N$ and $s<t$. In turn that holds because of Proposition \ref{lim}.

Consider any limit point\footnote{$\Bbb Q$ here can be replaced by any countable dense set in $\Bbb R$, in particular any set that contains a fixed finite set of times.} of $(Z^\epsilon_{s,t})_{s,t \in \Bbb Q}$ as $\epsilon \to 0$, which is a probability measure on the product $\sigma$-algebra of $\mathcal M(\Bbb R^2)^{\{(s,t)\in\Bbb Q^2: s\le t\}}.$ Such limit points exist by Kolmogorov's extension theorem and a diagonal argument. Indeed one can take the projective limit of any consistent family of finite-dimensional marginal laws which are themselves limit points of the finite-dimensional marginals $(Z^\epsilon_{s_1,t_1},...,Z^\epsilon_{s_m,t_m})$ as $\epsilon\to 0$.

Now we need to show that limit point does indeed satisfy all conditions in Proposition \ref{mr}, which will uniquely identify that limit point. Condition (1) is satisfied because it is already true in the prelimit. Condition (3) holds by Proposition \ref{lim} and the uniform integrability implied by the same proposition.

It remains to prove Condition (2). Because the propagator equation \eqref{propa} holds in the prelimit, this will be straightforward. Specifically, given the the fact that $\Bbb E[ (Z^\epsilon_{s,t},f)^n]$ remains bounded as $\epsilon \to 0$ for all $f\in C_c(\Bbb R^2)$, we have uniform integrability, and now we just need to show that $$\limsup_{\delta\to 0} \limsup_{\epsilon\to 0} \Bbb E \bigg[ \bigg(\int_{\Bbb R^4 } f(x,z) \delta^{-1} \psi(\delta^{-1}(y-w))Z^\epsilon_{t,u}(dy,dz) Z^\epsilon_{s,t}(dx,dw) - \int_{\Bbb R^2} f(x,z) Z^\epsilon_{s,u}(dx,dz)\bigg)^2 \bigg]=0.$$
In the prelimit, it is already clear that $Z^\epsilon_{s_j,t_j}$ are independent if $(s_j,t_j)$ are disjoint. Thus for $s<t<u$, we claim that 
\begin{align}
    \notag \lim_{\epsilon\to 0} \;&\Bbb E \bigg[ \bigg(\int_{\Bbb R^4 } f(x,z) \delta^{-1} \psi(\delta^{-1}(y-w))Z^\epsilon_{t,u}(dy,dz) Z^\epsilon_{s,t}(dx,dw) - \int_{\Bbb R^2} f(x,z) Z^\epsilon_{s,u}(dx,dz)\bigg)^2 \bigg] \\&=  \mathbf E \bigg[ \bigg(\int_{\Bbb R^4 } f(x,z) \delta^{-1} \psi(\delta^{-1}(y-w))U_{t,u}(dy,dz) U_{s,t}(dx,dw) - \int_{\Bbb R^2} f(x,z) U_{s,u}(dx,dz)\bigg)^2 \bigg]\label{limmo}
\end{align}
    where $U_{s,t}$ are the propagators associated to the limiting equation (Equation \eqref{she} but with the noise coefficient of $\sigma_p$), not necessarily defined on the same probability space as $Z^\epsilon$. To prove \eqref{limmo}, multiply out the expectation and we obtain 
    \begin{align*}
        \bigg(\int_{\Bbb R^4 } &f(x,z) \delta^{-1} \psi(\delta^{-1}(y-w))Z^\epsilon_{t,u}(dy,dz) Z^\epsilon_{s,t}(dx,dw) - \int_{\Bbb R^2} f(x,z) Z^\epsilon_{s,u}(dx,dz)\bigg)^2 \\ &= \bigg( \int_{\Bbb R^4} f(x,z) \bigg[ \delta^{-1} \psi(\delta^{-1} (y-w)) - \delta_0(y-w) \bigg]Z^\epsilon_{t,u}(y,z) Z^\epsilon_{s,t}(x,w)dydzdxdw\bigg)^2 \\&= \int_{\Bbb R^8} \prod_{j=1,2} f(x_j,z_j) \bigg[ \delta^{-1}\psi(\delta^{-1} (y_j-w_j)) - \delta_0(y_j-w_j)\bigg] \prod_{j=1,2} Z^\epsilon_{t,u}(y_j,z_j) Z^\epsilon_{s,t}(x_j,w_j) dy_jdz_jdx_jdw_j,
    \end{align*} 
    where $\delta_0$ is the Dirac mass (which should not be confused with the real number $\delta>0$), and we used \eqref{propa} in the second line. Taking expectation of the above, we will obtain an expression of the form \begin{equation}\label{hn}\int_{\Bbb R^4}  \Bbb E \bigg[ \mathcal H_\epsilon(x_1,w_1,x_2,w_2) \prod_{j=1,2} \mathcal Z^N_{s,t} (x_j,w_j)dx_jdw_j \bigg] 
    \end{equation}
    where for $x_j,w_j \in \Bbb R$ one has $$\mathcal H_\epsilon(x_1,w_1,x_2,w_2):= \int_{\Bbb R^4} \prod_{j=1,2} f(x_j,z_j) \bigg[ \delta^{-1}\psi(\delta^{-1} (y_j-w_j)) - \delta_0(y_j-w_j)\bigg] \Bbb E\bigg[\prod_{j=1,2} Z^\epsilon_{t,u}(dy_j,dz_j)\bigg].$$
    It follows from Proposition \ref{lim} that, as $\epsilon\to 0$, the functions $\mathcal H_\epsilon$ are uniformly bounded and converge uniformly on compacts of $\Bbb R^4$ to the function $$\mathcal H_\infty (x_1, w_1,x_2,w_2):= \int_{\Bbb R^4} \prod_{j=1,2} f(x_j,z_j) \bigg[ \delta^{-1}\psi(\delta^{-1} (y_j-w_j)) - \delta_0(y_j-w_j)\bigg] \mathbf E\bigg[\prod_{j=1,2} U_{t,u}(y_j,z_j)\bigg]\prod_{j=1,2} dy_jdz_j,$$ where $U_{s,t}$ are the propagators for \eqref{she} (not necessarily defined on the same probability space). This limit holds thanks to the presence of the integral against the smooth bounded function $f$. 
    
    We can say therefore that \eqref{hn} converges as $\epsilon\to 0$ to the quantity 
    \begin{equation}\label{hn2}
        \int_{\Bbb R^4} \mathcal H_\infty (x_1,w_1,x_2,w_2) \mathbf E \bigg[ \prod_{j=1,2} U_{s,t} (x_j,w_j) \bigg] \prod_{j=1,2} dx_jdw_j\end{equation}
        which (by following the same logic as for the prelimit above) is the same as $$\mathbf E \bigg[ \bigg(\int_{\Bbb R^4 } f(x,z) \delta^{-1} \psi(\delta^{-1}(y-w)) U_{t,u}(y,z) U_{s,t}(x,w)dydzdxdw - \int_{\Bbb R^2} f(x,z) U_{s,u}(x,z)dxdz\bigg)^2 \bigg].
    $$ This proves \eqref{limmo}. Then taking $\delta\to 0$ in \eqref{limmo} and using the continuity and uniform moment bounds of $U_{s,t}$ yields the claim.
\end{proof}

\section{Proof of Theorem \ref{mr3}}\label{s4}

The initial discussion here will parallel the observations that were first made in \cite{dom} and \cite{DDP23}. Henceforth we consider the propagators $\mathcal Z^\epsilon_{s,t}(x,y)$ of the equation given in Theorem \ref{mr3}. Already, these almost surely satisfy the propagator equation 
\begin{equation}\label{propa1}\mathcal Z^\epsilon_{s,u} (x,z) = \int_{\Bbb R} \mathcal Z_{s,t}^\epsilon(x,y)\mathcal Z_{t,u}^\epsilon(y,z)dy,\;\;\;\;\;\;\;\;\;\;\;\; s<t<u, \;\;\;\;\;\;\;\;\;\;x,z\in\Bbb R.
\end{equation}
To prove this, one can show that both sides of the equation satisfy the Duhamel form of the equation written in Theorem \ref{mr3}, for which \textit{pathwise uniqueness} is well-known \cite{Wal86} for any \textit{fixed} $\epsilon>0$. We do not write the calculation here, but the details of a similar calculation can be found in e.g. \cite[Theorem 3.1(vii)]{AKQ14b} or \cite[Lemma 3.12]{A}.

The bulk of this section will be devoted to calculating the limit of the moments as $\epsilon\to 0$ of the processes $\mathcal Z^\epsilon_{s,t}$, since that is what is needed to apply Proposition \ref{mr}.

\begin{defn}
    Define $\theta:= 1-\|\varphi\|_{L^2(\Bbb R)}^2 = 1-\varphi *\varphi (0)>0$, where $\varphi$ is as in Theorem \ref{mr3}. Define the following second-order elliptic operator on $C^2(\Bbb R^n)$: $$L^{(n,\epsilon)}_{\mathrm{Cen}}f(y_1,...,y_n):= \frac12 \sum_{i,j=1}^n \big( \theta \delta_{i,j}+\epsilon \Phi_\epsilon(y_i-y_j)\big)\cdot (\partial_{ij}f)(y_1,...,y_n).$$ 
    Denote by $ \mathbf E^{(x_1,...,x_n)}_{\mathrm{Cen}(\epsilon)}$ the expectation with respect to the Markov process $X$ on $\Bbb R^n$ with generator $L^{(n,\epsilon)}_{\mathrm{Cen}}$.
\end{defn}

We remark trivially that the Markov process $X$ with generator $L^{(n,\epsilon)}_{\mathrm{Cen}}$ is a martingale, thus the subscript refers to it being ``\textbf{cen}tered."

\begin{lem}\label{sing}
    Let $\Phi_\epsilon := \varphi_\epsilon * \varphi_\epsilon$, with $\varphi_\epsilon$ as in Theorem \ref{mr3}. Then for $\psi:\Bbb R\to \Bbb R$ bounded and continuous, and $x_1,...,x_n\in\Bbb R$, we have that$$\int_{\Bbb R^n} \Bbb E\bigg[ \prod_{j=1}^n \mathcal Z_{0,t}^\epsilon (x_j, y_j) \bigg]\prod_{j=1}^n \psi(y_j) dy\cdots dy_n = \mathbf E^{(x_1,...,x_n)}_{\mathrm{Sing}(\epsilon)} [ e^{  \sum_{1\le i<j\le k}\int_0^t \Phi_\epsilon (X^i_s-X^j_s)ds}\prod_{j=1}^n \psi(X^j_s)].$$Here the expectation on the right side is with respect to a diffusion process on $\Bbb R^n$ whose generator is $$L^{(n,\epsilon)}_{\mathrm{Sing}}f(y_1,...,y_n):= L^{(n,\epsilon)}_{\mathrm{Cen}}f(y_1,...,y_n) + \sum_{i\ne j} \epsilon^{1/2} \Phi_\epsilon(y_i-y_j) (\partial_i f)(y_1,...,y_n). $$
\end{lem}

Note here that the mechanism that will give rise to the KPZ limit is completely different than that of Theorem \ref{mr2}. Indeed in the proof of that theorem, we had $\Phi_\epsilon^{(2p-1)}$ in the exponent. It\^o formula with Girsanov thus gave rise to a process where the exponent had a term of the form $\epsilon^{4p-1} (\Phi_\epsilon^{(2p-1)})^2 \approx \sigma_p^2 \delta_0,$ which is the mechanism giving rise to the strange noise coefficient of $\sigma_p^2$ there. In contrast, here we have no derivatives on $\Phi_\epsilon$ and thus no It\^o formula. Rather the second-order part of the generator itself ``degenerates slightly" near the boundaries of the Weyl chamber, and these ``\textbf{sing}ularities" are what will give rise to the strange noise coefficient of $\gamma_{\mathrm{ext}}^2$ for the limiting SPDE.

\begin{proof} The proof is based on observations made in \cite{war}. Consider the diffusion in random environment given by \begin{equation}\label{sde}dX^i_s = \epsilon^{1/2} \eta^\epsilon(s,X^i_s)ds + \theta^{1/2} dW^i_s
\end{equation}
where $1\le i\le n$ and the environment $\eta^\epsilon$ is fixed (quenched), and $W^i$ are independent Brownian motions independent of $\eta$. Such a process makes sense, in fact the entire flow of diffeomorphisms for the SDE given $\eta^\epsilon$ is well-posed, see Theorem 4.5.1 of the monograph \cite{Kun94b}. As proved in \cite[Proposition 2.1]{ew6}, the Kolmogorov forward equation for each $X^i$ may be written as the It\^o-Walsh solution\footnote{This equation looks misleading because the It\^o-Stratonovich correction is not a constant multiple of $v$ but rather it is given by $\frac12 (1-\theta) \partial_x^2 v(t,x),$ which explains why the viscosity in the It\^o equation is $\frac12$ instead of $\frac12\theta$.} to 
\begin{equation}\label{4}\partial_t v(t,x) = \frac12 \partial_x^2 v(t,x) + \epsilon^{1/2} \partial_x \big( v(t,x) \eta^\epsilon(t,x)\big).
\end{equation}
If we average out the environment $\eta^\epsilon$ then the $n$-point motion $(X^1,...,X^n)$ is a diffusion in $\Bbb R^n$ with generator $L^{(n,\epsilon)}_{\mathrm{Cen}}$, more precisely if we let $v_{s,t}(x,y)$ denote the propagators of \eqref{4}, we have \begin{equation}\label{moma}\int_{\Bbb R^n} \Bbb E[ \prod_{j=1}^n v_{0,t}(x_j,y_j)] \prod_{j=1}^n \phi(y_j)dy\cdots dy_n = \mathbf E_{\mathrm{Cen}(\epsilon)}^{(x_1,...,x_n)} \big[ \prod_{j=1}^n \phi(X^j_t)\big].\end{equation}
This immediately follows from \eqref{sde} since  $\Bbb E[ (X^i_{t+dt}-X^i_t) (X^j_{t+dt}-X^j_t)| \mathcal F_t] = (\epsilon \Phi_\epsilon(X^i_t-X^j_t) +\theta \delta_{i,j})dt$, so that $(X^1,...,X^n)$ solves the martingale problem for $L^{(n,\epsilon)}_{\mathrm{Cen}}$; we refer to the discussion of \cite[(2.23)-(2.24)]{ew6} if one desires a more explicit calculation. As first observed in \cite{bld}, note by simple calculus that $u(t,y)$ from Theorem \ref{mr3} is related to $v(t,x)$ above, by the formula $$u(t,y) = e^{\frac12 \epsilon^{-1}t + \epsilon^{-1/2}y } v(t,y+\epsilon^{-1/2}t),$$ where the equality is in law as space-time processes, but it can be made pathwise by replacing $\eta(t,x)$ by $\eta(t,x-\epsilon^{-1/2}t)$ in \eqref{4} which preserves the law of the noise. 
The exponential part $e^{\frac12 \epsilon^{-1}t + \epsilon^{-1/2}y }$ can be interpreted as a change of measure by using the martingale $e^{\epsilon^{-1/2}(X^1+...+X^n) - \frac{\epsilon^{-1}}2 \langle X^1+...+X^n\rangle}.$ More precisely, recall that $\langle X^i,X^j\rangle_t = \epsilon \int_0^t \Phi_\epsilon(X^i_s-X^j_s)ds$, then use \eqref{moma} and then go through with this change of measure, and we will obtain that \begin{align*}\int_{\Bbb R^n} \Bbb E[ \prod_{j=1}^n &\mathcal Z^\epsilon_{0,t}(x_j,y_j)] \prod_{j=1}^n \phi(y_j)dy\cdots dy_j \\&= e^{\frac{n}2 \epsilon^{-1}t} \mathbf E_{\mathrm{Cen}(\epsilon)}^{(x_1,...,x_n)} \big[e^{\epsilon^{-1/2} (X^1_t +...+X^n_t)} \prod_{j=1}^n \phi(X^j_t+\epsilon^{-1/2}t)\big]\\&=  \mathbf E^{(x_1,...,x_n)}_{\mathrm{Sing}(\epsilon)} [ e^{  \sum_{1\le i<j\le k}\int_0^t \Phi_\epsilon (X^i_s-X^j_s)ds}\prod_{j=1}^n \phi(X^j_s)]
\end{align*}
as required for the lemma.
\end{proof}

\begin{lem}\label{CAF2}
    Under $\mathbf E^{(x_1,...,x_n)}_{\mathrm{Cen}(1)},$ each coordinate difference $X^i-X^j$ is a Markov process on $\Bbb R$ with generator $G_{\mathrm{diff}}:= (1- \varphi*\varphi(x))\partial_x^2$.
    Let $\mathbf E_{\mathrm{diff}}^x$ denote expectation with respect to that Markov process, started from $x\in\Bbb R$. There exists some universal constant $C>0$, such that uniformly over all smooth functions $\phi\ge 0$ with $\int_{\Bbb R} \phi =1,$ and all $s,t ,q\ge 0$, one has 
    \begin{align*}
        \sup_{x\in\Bbb R} \mathbf E^x_{\mathrm{diff}}\bigg[\bigg|\int_s^t \phi(Y_s)ds\bigg|^q\bigg]^{1/q} &\leq Cq^{1/2} |t-s|^{1/2}\\
        \sup_{x\in\Bbb R} \mathbf E^x_{\mathrm{diff}}[e^{q \int_0^t \phi(Y_s)ds}] &\leq Ce^{Cq^2 t}.
    \end{align*}
\end{lem}

\begin{proof}The first statement follows by direct calculation. For the two estimates, first note by standard heat kernel estimates for elliptic diffusions \cite[Theorem 1]{Aro} that $\mathbf E^x_{\mathrm{diff}} [ \phi(Y_t) ] \leq C \|\phi\|_{L^1(\Bbb R)} t^{-1/2},$ where $C$ is a universal constant independent of $t,x,$ and $\phi$. From here, the proof follows using exactly the same arguments as in the proof of Lemma \ref{CAF}, iteratively using the Markov property to get $m^{th}$ moment bounds, but replacing the Brownian motion $B$ there by the process $Y$ here.
\end{proof}

\begin{lem}\label{rescale}
    For any bounded measurable $F:C([0,T],\Bbb R^n)\to\Bbb R$ we have that $$\mathbf E^{(\epsilon x_1,...,\epsilon x_n)}_{\mathrm{Cen}(\epsilon)} \big[ F\big((X_t)_{t\in [0,T]}\big)\big] = \mathbf E^{(x_1,...,x_n)}_{\mathrm{Cen}(1)} \big[ F\big( (\epsilon X_{\epsilon^{-2} t})_{t\in [0,T]}\big) \big].$$ 
\end{lem}

The proof is clear from the expression for the generator $L^{(n,\epsilon)}_{\mathrm{Cen}}$.

\begin{prop}[Limit of the moments]\label{lim2}
    Let $\phi: \Bbb R\to \Bbb R$ be bounded and continuous. For all $n\in\Bbb N,$ and all $(x_1,...,x_n)\in\Bbb R^n$, we have $$\lim_{\epsilon\to 0} \int_{\Bbb R^n } \Bbb E \bigg[ \prod_{j=1}^n \mathcal Z_{0,t}^\epsilon (x_j, y_j) \bigg] \prod_{j=1}^n \phi(y_j) dy\cdots dy_n = \mathbf E_{\mathrm{BM}^{\otimes n}}^{(x_1,...,x_n)} \bigg[ e^{\gamma_{\mathrm{ext}}^2  \sum_{1\le i<j\le n} L_0^{W^i-W^j}(t) } \prod_{j=1}^n \phi(W^j_t)\bigg] .$$
    Here $\gamma_{\mathrm{ext}}^2$ is as in Theorem \ref{mr3}. The above convergence is uniform on compact subsets of $(x_1,...,x_n)$.
\end{prop}

The proof will be done using the diffusion representation of the moments from Lemma \ref{sing}. Section 2 of \cite{dom} suggests one method of calculating the $\epsilon\to 0$ behavior of that diffusion using previously known results and Taylor expansions, though they only consider $n=2$ and they do not write out the proof explicitly. We give a different ODE-based argument here (based on simplified ideas from \cite{DDP23,DDP23+}) that also keeps the discussion more self-contained.

\begin{proof}
    It is clear by Lemma \ref{sing} that we require a joint invariance principle for the entire collection of processes given by $$\bigg( \big(X^1_t,...,X^n_t\big) , \big(   \int_0^t \Phi_\epsilon(X^i_s-X^j_s) ds\big)_{1 \le i<j\le n}\bigg)_{t\in [0,T]},$$
    under the measures $\mathbf E^{(x_1,...,x_n)}_{\mathrm{Sing}(\epsilon)}$ as $\epsilon\to 0$. Specifically, we will now show that these processes have a \textbf{joint} limit in distribution given by $$\bigg( \big( W^1_t,..., W^n_t\big), \big( \gamma_{\mathrm{ext}}^2 L_0^{W^i-W^j}(t)\big)_{1\le i<j\le n}\bigg)_{t\in [0,T]} ,$$ where $(W^1,...,W^n)$ is a standard Wiener process and $L^X_0$ is the local time process of the semimartingale $X$. To prove this, the main idea will be to first use Girsanov to ``reduce" estimates for $\mathbf E^{(x_1,...,x_n)}_{\mathrm{Sing}(\epsilon)}$ to estimates for the simpler process $\mathbf E^{(x_1,...,x_n)}_{\mathrm{Cen}(\epsilon)}$, then in turn use the rescaling property of Lemma \ref{rescale} to reduce all estimates to the $\epsilon$-independent process $\mathbf E^{(x_1,...,x_n)}_{\mathrm{Cen}(1)}$, for which Lemma \ref{CAF2} already yields very strong bounds.

    Apply Girsanov, and from the explicit expressions for the generators $L^{(n,\epsilon)}_{\mathrm{Sing}}$ and $L^{(n,\epsilon)}_{\mathrm{Cen}}$ we obtain the existence of a martingale $D$ such that
    \begin{equation}\label{g1}\mathbf E^{(x_1,...,x_n)}_{\mathrm{Sing}(\epsilon)} [F(X)] = \mathbf E^{(x_1,...,x_n)}_{\mathrm{Cen}(\epsilon)}\big[ e^{D_t-\frac12 \langle D\rangle_t} F(X)\big]
    \end{equation}
    for all bounded measurable $F:C([0,t],\Bbb R^n)\to\Bbb R,$ where one can check that the martingale $D$ satisfies 
    \begin{equation}\label{g2}\langle D\rangle _t\leq C\sum_{1\le i<j\le n}\int_0^t \delta_\epsilon(X^i_s-X^j_s)ds
    \end{equation}
    for some absolute constant $C>0$ and some function $\delta\in C_c^\infty(\Bbb R)$ with $\int \delta=1$ and $\delta_\epsilon(y):=\epsilon^{-1}\delta(\epsilon^{-1}y).$

    For any continuous martingale $G$, if $\Bbb E [ e^{ 2\langle G\rangle_t} ]<\infty$ for all $t\ge 0$, then by Novikov's condition $e^{ 2G_t-2\langle G\rangle_t} $ is a martingale, thus we have the bound \begin{equation}\label{g0}\Bbb E[e^{G_t}] = \Bbb E[ e^{G_t- \langle G\rangle_t +\langle G \rangle_t}] \leq \Bbb E[e^{2G_t -2\langle G\rangle_t }]^{1/2} \Bbb E [ e^{ 2\langle G\rangle_t} ]^{1/2} = \Bbb E [ e^{ 2\langle G\rangle_t} ]^{1/2}.
    \end{equation}
    By Lemmas \ref{CAF2} and \ref{rescale} one has that $$\sup_{\substack{\epsilon\in (0,1]\\(x_1,...,x_n)\in\Bbb R^n}} \mathbf E_{\mathrm{Cen}(\epsilon)}^{(x_1,...,x_n)} \big[ e^{ q  \int_0^T \delta_\epsilon (X^i_s-X^j_s)ds} \big] \leq Ce^{q^2 CT} <\infty,$$ where $C$ is some universal constant. By H\"older's inequality together with \eqref{g1} and \eqref{g2} and \eqref{g0}, the last expression easily implies that if $F: C[0,T]\to\Bbb R$ is any measurable path functional, then \begin{equation}
        \label{sis2} \mathbf E^{(x_1,...,x_n)}_{\mathrm{Sing}(\epsilon)} [ F(X)] \le C \mathbf E_{\mathrm{Cen}(\epsilon)}^{(x_1,...,x_n)} [ F(X) ^{2}]^{1/2}
    \end{equation}where $C$ is an \textit{absolute constant} independent of $\epsilon\in (0,1]$ and $F$ (specifically $C$ is an upper bound, uniform over all $\epsilon$, for the $L^2$ norm of the Radon-Nikodym derivative of $\mathbf E^{(x_1,...,x_n)}_{\mathrm{Sing}(\epsilon)}$ with respect to $\mathbf E^{(x_1,...,x_n)}_{\mathrm{Cen}(\epsilon)}$). This will be a very useful fact going forward.

    Now notice that under $\mathbf E_{\mathrm{Cen}(\epsilon)}^{(x_1,...,x_n)},$ the process $X$ is a martingale and by explicit calculation one has $\partial_t \langle X^i\rangle_t \leq C$ for some deterministic constant $C$, thus by Burkholder-Davis-Gundy, one has the estimate $\mathbf E_{\mathrm{Cen}(\epsilon)}^{(x_1,...,x_n)}[ |X_t^i-X_s^j|^q]^{1/q} \leq C|t-s|^{1/2},$ from which \eqref{sis2} immediately implies \begin{equation}\label{g3}\mathbf E_{\mathrm{Sing}(\epsilon)}^{(x_1,...,x_n)}[ |X_t^i-X_s^j|^q]^{1/q} \leq C|t-s|^{1/2}.
    \end{equation}
    Now notice that Lemmas \ref{CAF2} and \ref{rescale} imply the bound $\mathbf E^{(x_1,...,x_n)}_{\mathrm{Cen}(\epsilon)}\big[ \big| \int_s^t \Phi_\epsilon(X^i_u-X^j_u) du \big|^q\big]^{1/q} \leq C|t-s|^{1/2}$, after which \eqref{sis2} immediately implies 
    \begin{equation}\label{g4}
        \mathbf E^{(x_1,...,x_n)}_{\mathrm{Sing}(\epsilon)}\big[ \big| \int_s^t \Phi_\epsilon(X^i_u-X^j_u) du \big|^q\big]^{1/q} \leq C|t-s|^{1/2}.
    \end{equation}
    Notice that \eqref{g3} and \eqref{g4} immediately give tightness in $C[0,T]$ of the full tuple of $\frac12 n(n+1)$ processes for which we want to prove the invariance principle. Now it remains to identify the limit points.

    First let us show that $(X^1,...,X^n)$ converges to a standard Brownian motion under $\mathbf E^{(x_1,...,x_n)}_{\mathrm{Sing}(\epsilon)}$. In the prelimit, note that $M_t^i := X_t^i - \epsilon^{1/2} \sum_{j\ne i} \int_0^t\Phi_\epsilon(X^j_s-X^j_s)ds$ is a martingale for $1\le i \le n$. Martingality is preserved by limit points as long as uniform integrability holds, thus by \eqref{g4} any limit point $(W^1,...,W^n)$ (of $(X^1,...,X^n)$ viewed under $\mathbf E^{(x_1,...,x_n)}_{\mathrm{Sing}(\epsilon)}$ as $\epsilon\to 0$) must be a martingale, thanks to the multiplying factor of $\epsilon^{1/2}$ on the drift part of the expression for $M^i_t$. Let us now compute its quadratic variation (with the intent of applying Levy's criterion to conclude Brownianity of $W$). From Dynkin's formula and the expression for the generator $L^{(n,\epsilon)}_{\mathrm{Sing}}$ it follows that $M^i_tM^j_t - \int_0^t \sigma_{ij}^\epsilon(X_s) ds$ is a martingale, where $\sigma^\epsilon_{ij}(x_1,...,x_n):= \theta \delta_{i,j} + \epsilon\Phi_\epsilon(x_i-x_j).$ Thanks to the multiplying factor of $\epsilon$ on the latter term, and \eqref{g4}, in the limit point we see that $W^i_t W^j_t -\delta_{i,j}t$ is a martingale. Thus the limit point $(W^1,...,W^n)$ is indeed a Brownian motion in $\Bbb R^n$ by Levy criterion.

    It remains to investigate the process-level behavior of $\big(   \int_0^\bullet \Phi_\epsilon(X^i_s-X^j_s) ds\big)_{1 \le i<j\le n}$ under the measures $\mathbf E^{(x_1,...,x_n)}_{\mathrm{Sing}(\epsilon)}$ as $\epsilon\to 0$. This is the more interesting part of this proof, as it is exactly where $\gamma_{\mathrm{ext}}$ will appear. Recall the ($\epsilon$-independent) Markov process $\mathbf E^x_{\mathrm{diff}}$ from Lemma \ref{CAF2}, and let $G_{\mathrm{diff}}:= (1 - \varphi*\varphi)\partial_x^2$ denote its generator. Let $\Psi$ denote the solution to the second-order ODE $$G_{\mathrm{diff}} \Psi = \varphi *\varphi, \;\;\;\;\;\; \lim_{y\to-\infty} \Psi(y)=0.$$
    It is immediate (e.g. from L'H\^opital's rule) that $$\lim_{y\to +\infty} y^{-1} \Psi(y) = \int_{\Bbb R} \Psi''(u)du=\int_\Bbb R \frac{\varphi * \varphi(u)}{1-\varphi * \varphi(u)} du = \gamma_{\mathrm{ext}}^2 ,$$ which is the key observation. In particular $\Psi_{(\epsilon)}(y):= \epsilon\Psi(\epsilon^{-1} y) $ converges uniformly on compacts to the function $x\mapsto \gamma_{\mathrm{ext}}^2 \max\{0,x\}.$ Define $\Gamma^{ij}_\epsilon (x_1,...,x_n):=\sum_{k\ne \ell} \Phi_\epsilon(x_k-x_\ell) \cdot \partial_k\big( \Psi_{(\epsilon)}(x_i-x_j)\big).$ Using Dynkin's formula, one checks that in the prelimit the process $$N^{ij}_t : = \Psi_{(\epsilon)} (X^i_t-X^j_t) -\int_0^t \Phi_\epsilon(X^i_s-X^j_s)ds - \epsilon^{1/2 } \int_0^t\Gamma^{ij}_\epsilon(X^1_s,...,X^n_s)ds $$ is a $\mathbf E^{(x_1,...,x_n)}_{\mathrm{Sing}(\epsilon)}$-martingale for $1\le i<j\le n$, where the middle term on the right side comes from the second-order part of the generator $L^{(n,\epsilon)}_{\mathrm{Sing}}$ because $G_{\mathrm{diff}}\Psi_{(1)} = \Phi_1$, and the last term comes from the first-order (drift) part of $L^{(n,\epsilon)}_{\mathrm{Sing}}$. Notice that $\sup_{\epsilon\in (0,1]} |\Psi_{(\epsilon)}'| \leq C$ for some absolute constant $C$, thus by \eqref{g4} one sees that the last term goes to zero in law in the topology of $C[0,T]$, under $\mathbf E^{(x_1,...,x_n)}_{\mathrm{Sing}(\epsilon)}$ as $\epsilon\to 0$.

    Let $(W^1,...,W^n, (\mathscr V^{ij})_{1\le i<j\le n})$ be a joint limit point of the full tuple of $\frac12n(n+1)$ processes given by $\big( (X^1_t,...,X^n_t) , ( \int_0^t \Phi_\epsilon(X^i_s-X^j_s) ds\big)_{1 \le i<j\le n}\big)_{t\in [0,T]}$ under the measures $\mathbf E^{(x_1,...,x_n)}_{\mathrm{Sing}(\epsilon)}$ as $\epsilon\to 0$. We already proved that the marginal law of $(W^1,...,W^n)$ must be a Brownian motion. Since martingality is preserved by limit points as long as uniform integrability holds, the martingality of $N^{ij}$ in the prelimit implies that $2\gamma_{\mathrm{ext}}^2 \max\{0, W^i_t-W^j_t\} - \mathscr V^{ij}_t$ is a continuous martingale for the limit point. By Tanaka's formula, this forces $\mathscr V^{ij}_t= \gamma_{\mathrm{ext}}^2 L_0^{W^i-W^j}(t)$, for \textbf{all} $1\le i<j\le n,$ thus uniquely identifying the law of the full tuple of $\frac12n(n+1)$ processes and completing the proof of the invariance principle. 

    Finally, we take the limit in the proposition statement. Using the invariance principle, and the uniform integrability guaranteed by Lemma \ref{CAF2} and Equation \eqref{sis2}, one has $$\lim_{\epsilon\to 0} \mathbf E^{(x_1,...,x_n)}_{\mathrm{Sing}(\epsilon)} \bigg[e^{ \sum_{1\le i<j\le n} \int_0^t \Phi_\epsilon(X^i_s-X^j_s) ds } \prod_{j=1}^n \phi(X^j_t) \bigg] = \mathbf E_{\mathrm{BM}^{\otimes n}}^{(x_1,...,x_n)} \bigg[ e^{\gamma_{\mathrm{ext}}^2  \sum_{1\le i<j\le n} L_0^{W^i-W^j}(t) } \prod_{j=1}^n \phi(W^j_t)\bigg].$$
    The convergence is uniform on compacts because allowing $x_i=x_i^\epsilon$ to vary within some bounded set as $\epsilon\to 0$ will not affect the invariance principle or the uniform integrability in any way.
\end{proof}

\begin{proof}[Proof of Theorem \ref{mr3}] The proof is nearly identical to that of Theorem \ref{mr2} given in Section 3, except that one uses Proposition \ref{lim2}, Equation \eqref{propa1}, and $\mathcal Z^\epsilon$ in place of Proposition \ref{lim}, Equation \eqref{propa}, and $Z^\epsilon$ respectively
\end{proof}

\bibliographystyle{alpha}
\bibliography{ref.bib}
\end{document}